\documentclass[12pt]{amsart}

\usepackage[margin=1in]{geometry}  
\usepackage{graphicx, array, blindtext}
\usepackage{amsthm, amsmath, amssymb,tikz, caption, subcaption}
\usepackage{comment}
\usepackage{multirow}
\usepackage{longtable}
\usepackage{booktabs}
\usepackage{caption, subcaption}
\usepackage{enumitem}

\usetikzlibrary{calc}

\tikzstyle{vx}=[shape=circle, minimum size=2mm, inner sep=0, fill=black]
\tikzstyle{whitevx}=[shape=circle, minimum size=2mm, draw=black, thick, fill=white, inner sep=0]
\tikzstyle{grayvx}=[shape=circle, minimum size=2mm, draw=black, thick, dash pattern = on 2pt off 1 pt, fill=gray!90, inner sep=0]

\tikzstyle{grayedge}=[line width=2.5pt,dash pattern=on 1pt off 1pt]
\tikzstyle{sgrayedge}=[line width=1.875pt,dash pattern=on 1pt off 1pt]
\tikzstyle{emphedge}=[ultra thick, dotted]

\usepackage{float}
\restylefloat{table}

\usepackage{color}   
\usepackage{hyperref}
\hypersetup{
    colorlinks=true, 
    linktoc=all,     
    linkcolor=blue,  
    filecolor=green
}

\usepackage{listings} 
\usepackage[toc,page]{appendix} 

\usepackage{todonotes} 
\usepackage{xspace}
\usepackage{bigstrut}

\theoremstyle{plain}
\newtheorem{thm}{Theorem}[section]
\newtheorem{lem}[thm]{Lemma}
\newtheorem{lemma}[thm]{Lemma}
\newtheorem{prop}[thm]{Proposition}

\newtheorem{cor}{Corollary}[thm]

\theoremstyle{definition}
\newtheorem{defn}[thm]{Definition}
\newtheorem{obs}[thm]{Observation}
\newtheorem{ques}[thm]{Question}
\newtheorem{cons}[thm]{Construction}

\theoremstyle{remark}

\newtheorem{case}{Case}[thm]

\newcommand{\indsat}[2]{\ensuremath{\operatorname{indsat}(#1,#2)}\xspace}
\newcommand{\sis}[2]{\ensuremath{{\operatorname{indsat}}^*(#1,#2)}\xspace}
\newcommand{\sat}[2]{\ensuremath{\operatorname{sat}(#1,#2)}\xspace}
\newcommand{\ex}[2]{\ensuremath{\operatorname{ex}(#1,#2)}\xspace}

\newcommand{\claw}{\ensuremath{K_{1,3}}\xspace}
\newcommand{\paw}{\ensuremath{K_{1,3}^+}\xspace}

\renewcommand{\bar}{\overline}

\newcommand{\cl}[1]{\lceil#1\rceil}

\newcommand{\cartprod}{\mathop{\scriptsize\Box}}

\usepackage{etoolbox}

\providetoggle{sub}
\settoggle{sub}{false}
\providetoggle{todo}
\settoggle{todo}{true}
\newtoggle{edit}
\togglefalse{edit}

\definecolor{wildstrawberry}{rgb}{1.0, 0.26, 0.64}

\begin{document}

\title{Graphs with induced-saturation number zero}

\author[S. Behrens]{Sarah Behrens}
\address[S. Behrens]{University of Nebraska-Lincoln}
\email{s-sbehren7@math.unl.edu}
\urladdr{http://www.math.unl.edu/~s-sbehren7}
\thanks{The research of the first author is supported in part by National Science Foundation grant DMS-0914815.}

\author[C. Erbes]{Catherine Erbes}
\address[C. Erbes]{Hiram College}
\email{erbescc@hiram.edu}
\thanks{The research of the second author is supported in part by National Science Foundation Grant DGE-0742434, UCD GK12 Transforming Experiences Project.}

\author[M. Santana]{Michael Santana}
\address[M. Santana, D. Yager, and E. Yeager]{University of Illinois at Urbana-Champaign}
\email[M. Santana]{santana@illinois.edu}
\urladdr[M. Santana]{http://math.uiuc.edu/~santana}
\thanks{The research of the third, fourth, and fifth authors is supported in part by National Science Foundation grant DMS 08-38434 ``EMSW21-MCTP: Research Experience for Graduate Students.''}

\author[D. Yager]{Derrek Yager}
\email[D. Yager]{yager2@illinois.edu}

\author[E. Yeager]{Elyse Yeager}
\email[E. Yeager]{yeager2@illinois.edu}
\urladdr[E. Yeager]{http://math.uiuc.edu/~yeager2}

\begin{abstract}
Given graphs $G$ and $H$, $G$ is $H$-saturated if $H$ is not a subgraph of $G$, but for all $e \notin E(G)$, $H$ appears as a subgraph of $G + e$.  While for every $n \ge |V(H)|$, there exists an $n$-vertex graph that is $H$-saturated, the same does not hold for induced subgraphs.  That is, there exist graphs $H$ and values of $n \ge |V(H)|$, for which every $n$-vertex graph $G$ either contains $H$ as an induced subgraph, or there exists $e \notin E(G)$ such that $G + e$ does not contain $H$ as an induced subgraph.  To circumvent this Martin and Smith \cite{MS} make use of trigraphs when introducing the concept of induced saturation and the induced saturation number of graphs. This allows for edges that can be included or excluded when searching for an induced copy of H, and the induced saturation number is the minimum number of such edges that are required.

In this paper, we show that the induced saturation number of many common graphs is zero.  Consequently, this yields graphs, instead of trigraphs, that are H-induced-saturated. We introduce a new parameter for such graphs, $\sis{n}{H}$, which is the minimum number of edges in an H-induced-saturated graph.  We provide bounds on $\sis{n}{H}$ for many graphs. In particular, we determine $\sis{n}{\paw}$ completely, and \iftoggle{sub}{$\sis{n}{\claw}$ within an additive constant of four.}{$\sis{n}{\claw}$ for infinitely many n.}
\end{abstract}

\maketitle

\iftoggle{sub}{}{
\tableofcontents
\listoftables
\listoffigures
\newpage}

\section{Background and Introduction}\label{sec:intro}

\subsection{Background and Definitions}

A well-known graph parameter is the \emph{saturation number}, defined for a graph $H$ and a whole number $n$, as the minimum number of edges in a graph $G$ on $n$ vertices such that $H$ is not a subgraph of $G$, but $H$ occurs if any edge is added to $G$. Formally,
\[\sat{n}{H}=\min\{|E(G)|:\text{G has~}n \text{~vertices}, H\not\subseteq G,\text{~and~}\forall e\notin E(G), H\subseteq G+e\}.\]

Determining the saturation number for a given graph $H$ has proven, in general, quite difficult. For more information on the saturation number, see the dynamic survey of Faudree, Faudree, and Schmitt \cite{FFS}.

A natural attempt at defining an induced variant of graph saturation would be to state that an $n$-vertex graph $G$ is $H$-induced-satured if $G$ is $H$-free and for all $e \notin E(G)$, $G + e$ contains $H$ as an induced subgraph.  Unfortunately, this is not always well-defined.  That is, there exist graphs $H$ and values of $n \ge |V(H)|$ for which every $n$-vertex graph $G$ either contains $H$ as induced subgraph, or there exists $e \notin E(G)$ such that $G+e$ is $H$-free.  A simple example is $n = 4$ and $H = \claw$.  

In this paper, we consider a variant of the saturation number introduced by Martin and Smith in 2012 that looks for induced copies of $H$, and considers deleting as well as adding edges. To create a well-defined parameter, Martin and Smith \cite{MS} make use of \emph{trigraphs}, objects also used by Chudnovsky and Seymour in their structure theorems on claw-free graphs \cite{CS}.

\begin{defn}
A  \textbf{trigraph} $T$ is a quadruple $(V(T), E_B(T), E_W(T), E_G(T))$, where $V(T)$ is the vertex set and the other three elements partition ${V(T)\choose 2}$ into a set $E_B(T)$ of black edges, a set $E_W(T)$ of white edges, and a set $E_G(T)$ of gray edges.  
These can be thought of as edges, nonedges, and potential edges, respectively. For any $e\in E_B(T)\cup E_W(T),$ let $T_e$ denote the trigraph where $e$ is changed to a gray edge, i.e. $T'=(V(T),E_B(T)-e,E_W(T)-e,E_G(T)+e)$.

A  \textbf{realization} of $T$ is a graph $G=(V(G),E(G))$ with $V(G)=V(T)$ and $E(G)=E_B(T)\cup S$ for some $S\subseteq E_G(T)$.  Let $\mathcal{R}(T)$ be the family of graphs that are a realization of $T$.

A trigraph $T$ is \textbf{$H$-induced-saturated} if no realization of $T$ contains $H$ as an induced subgraph, but $H$ occurs as an induced subgraph of some realization whenever any black or white edge of $T$ is changed to gray. Formally,

\begin{align*}
\indsat{n}{H}=\min\{|E_G(T)|:&|V(T)|=n, \forall G\in \mathcal{R}(T), H\not\leq G,\\
&\text{~and~}\forall e\in E_B(T)\cup E_W(T), H\leq G' \\
&\text{where~} G'\in\mathcal{R}(T_e)\}\\
\end{align*}


The \textbf{induced saturation number} of a graph 
$H$ with respect to $n$, written $\indsat{n}{H}$, is the minimum number of gray edges in an $H$-induced-saturated trigraph with $n$ vertices.

Notice that a trigraph with $E_G(T)=\emptyset$ has a unique realization, so if $\indsat{n}{H}=0$, there is a graph $G$ that has no induced copy of $H$ yet adding or removing any edge creates an induced copy of $H$. We will call such a graph \textbf{$H$-induced-saturated}.

The {\bf complement} of a trigraph $T$, denoted $\overline{T}$, is the trigraph with $V(\bar T)=V(T)$, $E_B(\overline{T}) = E_W(T)$, $E_W(\overline{T}) = E_B(T)$, and $E_G(\bar T)=E_G(T)$.
\end{defn}

\subsection{Notation} For graphs $G$ and $H$, we let $G \cup H$ denote the disjoint union, $G \vee H$ denote the join, and $G \cartprod H$ denote the Cartesian product of the two graphs.  
A trivial component of a graph is an isolated vertex.
For a graph $G$, we use $n(G)$ for the number of vertices and $e(G)$ for the number of edges in $G$. 
We let $P_n$ denote the path on $n$ vertices and $C_n$ the cycle on $n$ vertices. $K_n$ is the complete graph on $n$ vertices, and for $k \geq 2$, $K_{a_1,\ldots,a_k}$ is the complete multipartite graph with parts of size $a_1,\ldots,a_k$.  $\paw$ is the paw, which is obtained by adding a single edge to $K_{1,3}$.  For a set $S \subseteq V(G)$, $G[S]$ is the subgraph of $G$ induced by $S$, and if $S=\{v_1, \ldots, v_p\}$, we will sometimes write $G[v_1, \ldots, v_p]$. 
For a vertex $v\in V(G)$, $N_G(v)$ (or $N(v)$, if $G$ is clear from context) is the set of neighbors of $v$ in $G$, and $N[v]=N(v) \cup \{v\}$. We use $\deg_G(v)$ or $\deg(v)$ to denote the degree of $v$, that is, $|N(v)|$. 
In a trigraph, the black (resp. gray) degree of a vertex is the number of black (resp. gray) edges incident to that vertex. 
We say a set $S$ of vertices \emph{dominates} $G$ if every vertex of $G-S$ is adjacent to some vertex in $S$, and we call $S$ a {\em dominating set}; if $S=\{v\}$, we say $v$ is a \emph{dominating vertex}.
We say a vertex $u$ \emph{dominates} $S$ if $u$ is adjacent to every vertex in $S$.  
Finally, for an integer $n$, we let $[n]=\{1,\ldots, n\}$.
Other notation will be defined as it is used, or see \cite{West} for any undefined terms. 

\subsection{Observations and Previous Results}

By definition, the only trigraphs on fewer than $n(H)$ vertices that are $H$-induced-saturated are those in which all edges are gray. Thus we will usually assume that $n \ge v(H)$ when we compute $\indsat{n}{H}$. 

The following theorem summarizes the results of Martin and Smith~\cite{MS}:
\begin{thm}\label{thm:sat}
Let $H$ be a graph.
\begin{itemize}
\item For all $n\geq v(H)$, $\indsat{n}{H}\leq\sat{n}{H}$. By~\cite{KT}, $\sat{n}{H} \in O(n)$, so in particular $\indsat{n}{H}\in O(n)$.
\item  For all $n\geq m\geq3$, $\indsat{n}{K_m}=\sat{n}{K_m}$. (Note that $\sat{n}{K_m}$ was determined by Erd\H{o}s, Hajnal, and Moon in \cite{EHM}.)
\item For all $n\geq m\ge 2$, and for $e \in E(K_m)$, $\indsat{n}{K_m-e}=0$. In particular, for all $n\geq3$, $\indsat{n}{P_3}=0$.
\item For all $n\geq4$, $\indsat{n}{P_4}=\left\lceil\frac{n+1}{3}\right\rceil$.
\end{itemize}
\end{thm}


\begin{obs}\label{thm:complement}
A trigraph $T$ is $H$-induced-saturated if and only if $\bar T$ is $\bar H$-induced-saturated. In particular, $\indsat{n}{H}=\indsat{n}{\overline{H}}$.
\end{obs}
\begin{proof} 
Suppose a trigraph $T$ has a realization $G$ such that $H$ is an induced subgraph of $G$. Then $\bar H$ is an induced subgraph of $\bar G$. Using the definition of $\bar T$, 
$\bar G$ is a representation of $\bar T$.
It follows that a trigraph $T$ is $H$-induced-saturated if and only if $\bar T$ is $\bar H$-induced-saturated.
\end{proof}


\subsection{Minimally $H$-induced-saturated Graphs}

In this paper we show that $\indsat{n}{H}$ is zero for several graphs, which as noted above, means that there exists a graph that is $H$-induced-saturated. 
This leads to the natural question: What is the minimum number of edges in such a graph? 

\begin{defn}
For a graph $H$ and whole number $n$ with $\indsat{n}{H}=0$, we define 
$$\sis{n}{H}:=\min\{e(G):v(G)=n \text{ and }G\text{ is }H\text{-induced-saturated}\}.$$

We say a graph $G$ on $n$ vertices with $\sis{n}{H}$ edges is {\bf minimally $H$-induced-saturated}.\end{defn}

By Observation \ref{thm:complement}, the \emph{maximum} number of edges in an $n$-vertex $H$-induced-saturated graph is 
${n \choose 2} - \sis{n}{\overline H}$.

In this paper we will show that the following graphs have induced-saturation number zero for $n$ sufficiently large: $\paw$, stars $K_{1,t}$, $C_4$, odd cycles,  some modifications of even cycles, and matchings. 
Additionally, we provide bounds on $\sis{n}{H}$ for graphs listed above.  In particular, we characterize the $\paw$-induced-saturated graphs, which in turn completely determines $\sis{n}{\paw}$.  We also determine $\sis{n}{K_{1,t}}$ within a factor of 2 and show that the upper bound is correct for$K_{1,3}$.  Finally, we introduce the induced-saturation number of a family of graphs and show that while every graph in a family may have induced-saturation number zero, the family itself could have positive induced-saturation number.


%
%

\section{The Paw}
In this section we provide a construction that shows $\indsat{n}{\paw} = 0$ for $n \ge 7$.  We then show that our construction characterizes all $\paw$-induced-saturated graphs, allowing us to completely determine $\sis{n}{\paw}$ for $n \ge 7$.  

This construction, given in Construction \ref{cons:paw} requires $n \ge 7$, and since Theorem \ref{paw_unique} will show that these are the only $\paw$-induced-saturated graphs, we deduce that $\indsat{n}{\paw}$ is nonzero for $n \in \{4,5,6\}$.  The exact values for such $n$ are provided in Table \ref{tab:paw}. 

Table~\ref{tab:paw} exhibits paw-induced-saturated trigraphs on $n$ vertices with only one gray edge for $n \in \{5,6\}$. Since $\indsat{n}{\paw}>0$, this establishes $\indsat{n}{\paw}=1$ for such $n$.

For $n=4$, Table~\ref{tab:paw} gives a 4-vertex, paw-induced-saturated trigraph with two gray edges. To show that $\indsat{4}{\paw} = 2$, we argue that any 4-vertex trigraph $T$ with only one gray edge is not \paw induced saturated. 

$T$ has at least two black edges, otherwise chaning a white edge to gray does not result in a realization with an induced \paw.  Now suppose $T$ has no white edges. Since it has precisely one gray edge, its black edges form $K_4-e$, and changing the black edge whose endpoints have black degree three to a gray edge does not result in a realization with an induced \paw.
Next, suppose $T$ has at least two white edges. Since \paw has precisely two nonedges, changing a black edge to gray does not result in a realization with an induced \paw., unless $T$ already had such a realization.  Therefore $T$ has precisely one white edge. If the gray edge of $T$ is incident to the white edge, then \paw is a realization, so the black edges induce $C_4$. Since $C_4 \not\subseteq \paw$, changing the white edge to gray does not create an induced \paw. 

\begin{table}[htdp]
\newcolumntype{S}{>{\centering\arraybackslash} m{.23\linewidth} }
\newcolumntype{T}{>{\centering\arraybackslash} m{.26\linewidth} }
\centering
\caption{Values of $\indsat{n}{\paw}$ for $4\leq n\leq6$ and trigraphs realizing those values}
\label{tab:paw}
\begin{tabular}{S|T}
\  & trigraph\\
\hline \\[-2ex] 
\ $\indsat{4}{\paw}=2$ & 
\begin{tikzpicture}
\node at (0,0) [vx](1){};
\node at (0,.75) [vx](0){};
\node at (-1.25,-.5) [vx](2){};
\node at (1.25,-.5) [vx](3){};
\draw (0)--(2)--(1)--(3)--(0);
\draw[grayedge] (0)--(1);
\draw[grayedge] (2)--(3);
\end{tikzpicture} \\
\hline \\[-2ex]
\ $\indsat{5}{\paw}=1$ & 
\begin{tikzpicture}
\node at (0,0) [vx](2){};
\node at (1,0) [vx](3){};
\node at (2,0) [vx](4){};
\node at (-1,1) [vx](0){};
\node at (-1,-1)[vx](1){};
\foreach \x in {2,3,4}{\draw (0)--(\x)--(1);}
\draw[grayedge] (0)--(1);
\end{tikzpicture}\\
\hline \\[-2ex]
\ $\indsat{6}{\paw}=1$ & 
\begin{tikzpicture}
\node at (0,0) [vx](2){};
\node at (1,0) [vx](3){};
\node at (2,0) [vx](4){};
\node at (3,0) [vx](5){};
\node at (-1,1) [vx](0){};
\node at (-1,-1)[vx](1){};
\foreach \x in {2,3,4,5}{\draw (0)--(\x)--(1);}
\draw[grayedge] (0)--(1);
\end{tikzpicture} \\
\end{tabular}
\end{table}

Having established $\indsat{n}{\paw}$ for small values of $n$, we now present our construction.

\begin{cons}\label{cons:paw}
Let $G$ be a graph with at most one trivial component, where each nontrivial component is complete multipartite, each with at least three parts, at most one of which contains only one vertex, and the remainder of which have order at least three. 
\end{cons}


\begin{prop}
The graph $G$ in Construction \ref{cons:paw} is \paw-induced-saturated. 
\end{prop}
\begin{proof} 
Since \paw is not an induced subgraph of a complete multipartite graph, $G$ contains no induced \paw.
Suppose we add an edge $xy$ such that $x$ and $y$ are in distinct components, say $F_x$ and $F_y$, respectively.  Since at least one of these components, say $F_x$, has at least three parts, $x$ is in some triangle $xab$ in $F_x$. 
Because $y$ is in a different component, $y$ is adjacent to $x$ but not $a$ or $b$. 
Thus $\{x,y,a,b\}$ induces a \paw.

Suppose we add an edge $xy$ such that $x$ and $y$ are in the same component. 
Then in particular, they are in the same part. 
This part has at least two vertices, so by construction it has at least three vertices; choose $z$ distinct from $x$ and $y$ from this part, and let $a$ be in another part of the component. 
Then $\{x,y,z,a\}$ induces a \paw.

Suppose we delete an edge $xy$. 
Then $x$ and $y$ were in different parts of one component, say $F$. 
As $F$ is complete multipartite with at least three parts, there exists a vertex $z$ in a third part of that component. Since at most one part has only one vertex, there is a vertex $a$ in the same part as either $x$ or $y$; say $x$. Then $\{x,y,z,a\}$ induces a \paw.
\end{proof}

\begin{cor}\label{indsat_paw}
For $n \ge 7$, $\indsat{n}{\paw}=0$.
\end{cor}

We now show that Construction~\ref{cons:paw} describes all \paw-induced-saturated graphs. 

\begin{thm}\label{paw_unique}
A graph is \paw-induced-saturated if and only if it is as described in Construction \ref{cons:paw}.
\end{thm}

To prove this theorem, we begin by making several observations.

\begin{lem}\label{lem:paw_obs}
Let $G$ be a \paw-induced-saturated graph. Then $G$ has the following properties:
\begin{enumerate}[label=(\alph*), ref=\ref{lem:paw_obs}(\alph*)]
\item\label{triangle} Every edge of $G$ is in a triangle.
\item\label{multipartite} The neighborhood of any vertex of $G$ is a complete multipartite graph.
\item\label{book} Given any non-isolated vertex $v \in V(G)$, there exists a (possibly empty) independent set $S=S(v)$ such that for every $x \in N(v)$, $S = N(x) \setminus N[v]$.
\end{enumerate}
\end{lem}
\begin{proof}
Lemma~\ref{triangle} holds because deleting any edge in $G$ creates an induced \paw. As a consequence, any vertex has degree either zero or at least two.

Since $G$ does not contain an induced $\paw$, the neighborhood of any vertex cannot contain an induced copy of $K_2 \cup K_1$. 
This is equivalent to the neighborhood being a complete multipartite graph. This gives us Lemma~\ref{multipartite}.

To prove Lemma~\ref{book}, suppose there exists $x \in N(v)$ that has a neighbor not in $N[v]$. (If no such $x$ exists, the claim holds with $S = \emptyset$.) 
Let $S:=N(x) \setminus N[v]$. If $G[S]$ has an edge $ss'$, then $G[v,x,s,s']$ is a paw. Since $G$ is \paw-induced-saturated, we conclude that $S$ is independent.

By Lemma~\ref{triangle}, there exists $y \in N(v) \cap N(x)$. 
If any element $s \in S$ is not adjacent to $y$, then $G[v,x,y,s]$ is a paw with $s$ as the pendant vertex. 
Therefore, $S \subseteq N(y)$, but also $N(y)\backslash N[v]\subseteq S$ or else we would have a paw.  Because $N(v)$ is complete multipartite by Lemma~\ref{multipartite}, every vertex in $N(v)\setminus\{x,y\}$ is adjacent to $x$ or $y$. 
By symmetry, we conclude that for every $z \in N(v)$, $N(z)\setminus N[v] = S$.
\end{proof}

We proceed to the proof of Theorem~\ref{paw_unique}.
\begin{proof}[Proof of Theorem~\ref{paw_unique}]
Let $G$ be a \paw-induced-saturated graph.  It is clear that $G$ has at most one nontrivial component, since adding an edge between two isolated vertices does not create an induced $\paw$.  We now show that every nontrivial component of $G$ is a complete multipartite graph.
Let $v$ be a non-isolated vertex in $G$ and let $S$ be the set given by Lemma~\ref{book}.
By Lemmas~\ref{multipartite} and \ref{book}, $G[N[v] \cup S]$ is a complete multipartite graph, with $v$ and $S$ sharing a part.
So, we need only show $N[v] \cup S$ is a component of $G$. If not, then there exists some vertex $s \in S$ with a neighbor $t \not\in N[v]\cup S$ since we have included the neighborhood of every $x\in N[v]$ and $S$ is an independent set. 
If there exists an edge $xy$ in $G[N(v)]$, then $G[x,y,s,t]$ is a paw, so $N(v)$ is an independent set.
This violates Lemma~\ref{triangle}.

Now, by Lemma~\ref{triangle}, every nontrivial component of $G$ must have at least three parts.
Next, we show that no part in any component of $G$ has order two, and any component has at most two parts of order one. Suppose $x$ and $y$ either make up a part of order two, or are each a part of order one in a component $F$. Then $\{x,y\}$ dominates $F\setminus\{x,y\}$, and so $x$ and $y$ do not appear together in an induced paw, so adding or deleting the edge $xy$ does not create an induced paw. Hence, $G$ being \paw-induced-saturated implies that it can be formed by Construction \ref{cons:paw}.
\end{proof}

\begin{cor}\label{sis paw}For $n \geq 7$, let $n\equiv r \mod 7$, where $0 \le r \le 6$. Then
\begin{equation*}\sis{n}{\paw} = \left\{
\begin{array}{cl}\displaystyle\frac{15}{7}n & \mbox{ if } r =0\\
15\lfloor n/7\rfloor+4(r-1) & \mbox{ if } r \neq 0\end{array}
\right..\end{equation*}
\end{cor}

\begin{proof}
Let $G$ be a minimally \paw-induced-saturated graph on $n$ vertices.
From Theorem \ref{paw_unique}, each nontrivial component of $G$ is a complete multipartite graph with at least three parts.
If some nontrivial component $F$ of $G$ has at least three parts, then we form a \paw-induced-saturated graph with strictly fewer edges
 by dropping edges between two of the parts and forming a single larger part.
Thus each nontrivial component of $G$ is tripartite.  

The number of edges of a complete tripartite graph on $m$ vertices with parts of size $s,t,$ and $m-(s+t)$ is given by $(m-[s+t])(s+t)+st$.   
Given the constraints $s\geq1, t\geq3$, and $m\geq t$, we
see that $(m-[s+t])(s+t)$ is minimized when $s+t$ is minimized, i.e. $s+t=4$; also $st$ is minimized when $s+t$ is minimized.
Therefore, $K_{1,3,m-4}$ obtains the smallest number of edges among all complete tripartite graphs on $m$ vertices.

Now, we may assume $G$ has components $F_0,F_1,\ldots,F_i$ with $n(F_0) \in \{0,1\}$
and for $i>0$, $F_i=K_{1,3,n_i-4}$, where $n(F_0)+\sum_{i=1}^kn_i=n$.
Then: 
$$e(G) = \sum_{i=1}^k e(F_i)= \sum_{i=1}^k (4n_i-13)= 
4n-13k-4v(F_0).$$

Clearly, this is minimized by taking $k$ as big as possible and, subject to this, $n(F_0)=1$.
That is, we take $k=\lfloor n/7 \rfloor$ and 
\begin{displaymath}v(F_0)=\left\{\begin{array}{rl}
0 & \mbox{ if 7 divides $n$}\\
1 & \mbox{ else.}
\end{array}
\right.
\end{displaymath}
\end{proof}
  
\begin{obs}
Given $H$ for which $\sis{n}{H}$ is defined for all sufficiently large $n$,
the function $\sis{n}{H}$ is not necessarily monotone in $n$. In particular, from Corollary~\ref{sis paw} we see for any integer $k \geq 2$, 
$\sis{7k}{\paw}<\sis{7k+2}{\paw}<\sis{7k-1}{\paw}$. 
This is a similarity between minimal induced saturation and saturation: as noted in~\cite{FFS}, the function $\sat{n}{H}$ is not necessarily monotone in $n$ for fixed $H$.
\end{obs}




%
%

\section{Stars}
 Recall that $K_{1,2} = P_3$, and $\indsat{n}{P_3} = 0$ for $n \ge 3$, as established in \cite{MS}.  In this section we provide a construction extending this result, to show that for fixed $k \ge 2$ and $n$ sufficiently large, $\indsat{n}{K_{1,k+1}}=0$.  Additionally, our construction, together with a simple argument, determines $\sis{n}{K_{1,k+1}}$ within a factor of two.  The case when $k = 2$, which refers to the graph $K_{1,3}$, commonly known as the claw, will be addressed in further detail in Section \ref{sec:claw}.

\begin{cons}\label{cons:stars}
Fix $k \ge 2$ and $n \ge 3^k$.  Let $z,R$ be positive integers such that $n = z3^k + R$ with $0 \le R < 3^k$.  Let $H$ be the graph $K_3^1 \cartprod K_3^2 \cartprod \cdots \cartprod K_3^k$, where $K_3^i$ denotes a single copy of $K_3$.  In other words, $V(H) = \{(\alpha_1,\dots, \alpha_k) : \alpha_i \in [3]\}$, and $(\alpha_1,\dots, \alpha_k)(\beta_1,\dots, \beta_k) \in E(H)$ iff $\sum \{i:\alpha_i\neq\beta_i\} = 1$.  Define $H'$ where $V(H')$ is  the disjoint union of $V(H)$ and $V(K_R)$, and $E(H')$ consists of $E(H),$ $E(K_R)$ and the edges between $H$ and $K_R$ satisfy: for each $v \in V(K_R)$, $v\alpha \in E(H')$ if and only if $\alpha = (\alpha_1,1,1,\dots,1)$, $\alpha_1 \in [3]$.  
We now define $G$ to be the disjoint union of $z-1$ copies of $H$ and a single copy of $H'$.
\end{cons}

\begin{prop}\label{prop:stars2}
The graphs in Construction \ref{cons:stars} are $K_{1,k+1}$-induced-saturated.
\end{prop}
\begin{proof}
Given fixed $n$ and $k$, let $G$ and $R$ be as defined in Construction \ref{cons:stars}.  Let $F$ denote the subgraph of $H'$ isomorphic to $H$.   Suppose to the contrary that $G$ contains an induced $K_{1,k+1}$ with center $x$.  Suppose first that $x$ is in a copy of $H$.  Since $x$ and its neighbors can be represented with $k$-dimensional vectors, by the Pigeonhole Principle, any $k+1$ neighbors of $x$ have two vectors which differ in exactly the same coordinate from each other.  Thus, $x$ cannot have $k+1$ neighbors which form an independent set, and $H$ is $K_{1,k+1}$-free.  

If $H'$ contains the induced $K_{1,k+1}$, then $x$ cannot be in the $K_R$ as the neighborhood of $x$ would be a clique.  So $x$ is in $F$.  If this induced $K_{1,k+1}$ contains no vertices from the copy of $K_R$, then the above argument produces  a contradiction.  Thus, this $K_{1,k+1}$ contains a vertex from the copy of $K_R$, and without loss of generality, we may assume that $x$ represented by $(1,1,\dots,1)$ in $F$.  Consequently, our $K_{1,k+1}$ has exactly one vertex in $K_R$, but then contains no vertices of the form $(\alpha_1,1,1,\dots,1)$ other than $x$.  Hence by pigeon hole, $x$ has at most $k-1$ other neighbors in our $K_{1,k+1}$ from $F$, a contradiction.   So $G$ is $K_{1,k+1}$-free.

It is clear that every vertex in a copy of $H$ (or in $F$) is the center of an induced $K_{1,k}$.  Thus, if we add an edge between two components of $G$, one component must be a copy of $H$, and we obtain an induced $K_{1,k+1}$.  Thus, it remains to consider adding an edge within a component.  Note that by the construction of $H'$, the only possible way to add an edge is within $F$, which is isomorphic to $H$.  So, it suffices to consider adding an edge to a copy of $H$.  Suppose we add the edge $uv$.  Without loss of generality, we may assume that $u$ is represented by $(1,1,\dots, 1)$.  Since $u$ and $v$ were not adjacent in $H$, their corresponding vectors must differ in at least two coordinates, say the first and second.  As a consequence, $v$ is adjacent to neither $y$ nor $w$, where $y \in \{(2,1,1,\dots,1), (3,1,1,\dots,1)\}$ and $w \in \{(1,2,1,1,\dots,1), (1,3,1,1,\dots,1)\}$.  Thus, $\{u,v,w,y\}$ is an induced $K_{1,3}$ centered at $u$.  To this set we add vertices $\alpha^3,\alpha^4,\dots, \alpha^k$, where $\alpha^i$ has all coordinates equal to 1 except that the $i$th coordinate is either 2 or 3.  This induces $K_{1,k+1}$.

Lastly, suppose we remove an edge $uv$.  There are three cases to consider.  The first case is if $uv$ is in a copy of $H$ (or in $F$).  Here, we may assume $u = (2,1,1,\dots,1)$ and $v = (3,1,1,\dots,1)$.  The second case is if both $u$ and $v$ are in $K_R$.  The last case is if only, say $v$, is in $K_R$.  Here, we may again assume that $u = (2,1,1,\dots,1)$.  In all three cases, $(1,1,\dots,1)$ together with $u,v,\alpha^2,\dots,\alpha^k$ as defined above, induce a $K_{1,k+1}$.  This completes the lemma.
\end{proof}

\begin{cor}\label{cor:stars}
For fixed $k \ge 2$ and $n \ge 3^k$, $\indsat{n}{K_{1,k+1}} = 0$.
\end{cor}


\begin{thm}\label{bound:stars}
For $n \geq 2\cdot 3^k$ and $k \geq 2$, there exist constants $c_1=c_1(k)$ and $c_2 = c_2(k)$ such that $n\frac{k}{2} - c_1 \le \sis{n}{K_{1,k+1}} \leq nk+c_2$.
\end{thm}
\begin{proof}
Given fixed $n$ and $k$, let $G$ and $R$ be as defined in Construction \ref{cons:stars}.

We  establish $e(G)$ by considering vertex degrees. The component $H'$ has at most $2\cdot 3^k$ vertices, and so (trivially) at most ${2\cdot 3^k \choose 2}$ edges. The remaining vertices, of which there are at most $n-3^k$, all have degree $2k$ for a contribution of at most $(n-3^k)k$ edges. All told, $e(G) \leq nk-k\cdot 3^k + {2\cdot 3^k \choose 2}$.

To show the lower bound, suppose that $G$ is a $K_{1,k+1}$-induced-saturated graph.  Let $S=\{x \in V(G) \colon \deg(x) \le k-1\}$. 
We claim that $|S| \le k$.

If $|S| > k$, then there exist $x,y \in S$ such that $xy \notin E(G)$.  Let $G'$ denote $G + xy$.  As $G$ was $K_{1,k+1}$-induced-saturated, $G'$ must contain an induced $K_{1,k+1}$, using the edge $xy$ with either $x$ or $y$ as the center of this $K_{1,k+1}$.  However, as both $x$ and $y$ are adjacent to at most $k - 1$ vertices in $G$, this cannot happen.  So $|S| \le k$, as claimed.

Observe:  $$
|E(G)| \ge  \frac{1}{2}\left(k(n - |S|) + \sum\limits_{x \in S} \deg(x)\right) \ge \frac{nk}{2} - \frac{k^2}{2}.$$

This establishes the lower bound.
\end{proof}

It worth noting that we can extend Construction \ref{cons:stars}, as any graph formed as a Cartesian product of exactly $k$ cliques, each of size at least three, is $K_{1,k+1}$-induced-saturated.

%
%

\section{The Claw}\label{sec:claw}

For sufficiently large $n$, Theorem \ref{bound:stars} states that the order of magnitude of $\sis{n}{K_{1,k+1}}$ is linear in $n$, and in particular, we know the coefficient within a factor of two.  In this section, we will determine the coefficient of $\sis{n}{K_{1,3}}$, which coincides with the upper bound given in Theorem \ref{bound:stars}.  Additionally, we will provide better constructions than than that in Construction \ref{cons:stars}, which will ultimately determine $\sis{n}{K_{1,3}}$ within an additive constant of four.

Values of $\indsat{n}{\claw}$ were determined for $4\leq n \leq 10$ by computer search\footnote{A program was written in C++ and is available at \url{http://www.math.unl.edu/~s-sbehren7/main/Data.html}.} and are listed in \hyperref[tab:claws]{Table~\ref*{tab:claws}}, along with trigraphs that achieve the minimum number of gray edges.
This together with Corollary \ref{cor:stars} determines $\indsat{n}{\claw}$ for all $n$.
We now turn our attention to $\sis{n}{\claw}$.

\begin{table}[htdp] 
\newcolumntype{V}[1]{>{\centering\arraybackslash} m{#1\linewidth} }
\centering
\begin{tabular}{|c||c|}
\hline 
\begin{tabular}{V{.205}|V{.2}}
$\indsat{4}{\claw}=3$ &
\begin{tikzpicture}[scale=.75]
\node at (1,1) [vx, scale=.75](0){};
\node at (0,0) [vx, scale=.75](1){};
\node at (2,0) [vx, scale=.75](2){};
\node at (3,.5) [vx, scale=.75](3){};
\draw[sgrayedge] (0)--(1)--(2)--(0);
\end{tikzpicture}
\\
\hline \\[-2ex]
$\indsat{5}{\claw}=3$ &
\begin{tikzpicture}[scale=.75]
\node at (1,1) [vx, scale=.75](0){};
\node at (0,0) [vx, scale=.75](1){};
\node at (2,0) [vx, scale=.75](2){};
\node at (3,.5) [vx, scale=.75](3){};
\node at (4,.5)[vx, scale=.75](4){};
\draw[sgrayedge] (0)--(2)--(3)--(0);
\draw (0)--(1)--(2);
\end{tikzpicture}
\\
\hline \\[-2ex]
$\indsat{6}{\claw}=3$ & 
\begin{tikzpicture}[scale=.75]
\node at (0,0) [vx, scale=.75](0){};
\node at (1,.75) [vx, scale=.75](1){};
\node at (1,-.75)[vx, scale=.75](2){};
\node at (2,0)[vx, scale=.75](3){};
\node at (3,0)[vx, scale=.75](4){};
\node at (4,0)[vx, scale=.75](5){};
\draw (0)--(1)--(3)--(2)--(0)--(3);
\draw[sgrayedge] (1)--(2)--(4)--(1);
\end{tikzpicture}
\\
\hline \\[-2ex]
$\indsat{7}{\claw}=2$&
\begin{tikzpicture}
\node at (0,0)[vx, scale=.75](0){};
\node at (1,0)[vx, scale=.75](1){};
\node at (2,0)[vx, scale=.75](2){};
\node at (.5,.5)[vx, scale=.75](3){};
\node at (1.5,.5)[vx, scale=.75](4){};
\node at (.5,-.5)[vx, scale=.75](5){};
\node at (1.5,-.5)[vx, scale=.75](6){};
\draw (0)--(3)--(4)--(2)--(6)--(5)--(0);
\draw (3)--(1)--(5);
\draw (4)--(1)--(6);
\draw[sgrayedge] (3)--(5);
\draw[sgrayedge] (4)--(6);
\end{tikzpicture}\\
\hline \\[-2ex]
$\indsat{8}{\claw}=2$& 
\begin{tikzpicture}
\node at (0,0)[vx, scale=.75](0){};
\node at (1,0)[vx, scale=.75](1){};
\node at (2,0)[vx, scale=.75](2){};
\node at (.5,.5)[vx, scale=.75](3){};
\node at (1.5,.5)[vx, scale=.75](4){};
\node at (.5,-.5)[vx, scale=.75](5){};
\node at (1.5,-.5)[vx, scale=.75](6){};
\draw (0)--(3)--(4)--(2)--(6)--(5)--(0);
\draw (3)--(1)--(5);
\draw (4)--(1)--(6);
\draw[sgrayedge] (3)--(5);
\draw[sgrayedge] (4)--(6);
\node at (2.75,0)[vx, scale=.75](7){};
\end{tikzpicture}
\\
\hline\\[-2ex]
$\indsat{9}{\claw}=0$& 
\begin{tikzpicture}[xscale=.75, yscale=.5]
\node at (0,0)[vx, scale=.75](0){};
\node at (1,.5)[vx, scale=.75](1){};
\node at (2,0)[vx, scale=.75](2){};
\node at (.5,1)[vx, scale=.75](3){};
\node at (1.5,1.5)[vx, scale=.75](4){};
\node at (2.5,1)[vx, scale=.75](5){};
\node at (0,2)[vx, scale=.75](6){};
\node at (1,2.5)[vx, scale=.75](7){};
\node at (2,2)[vx, scale=.75](8){};
\draw (0)--(1)--(2)--(0);
\draw (3)--(4)--(5)--(3);
\draw (6)--(7)--(8)--(6);
\draw (0)--(3)--(6)--(0);
\draw (1)--(4)--(7)--(1);
\draw (2)--(5)--(8)--(2);
\end{tikzpicture}
\\
\end{tabular}
&\begin{tabular}{V{.22}|V{.19}}
\multirow{5}{*}{$\indsat{10}{\claw}=0$} & 
\begin{tikzpicture}[xscale=.75, yscale=.5]
\clip (-.25,-.25) rectangle (3.5, 2.5cm + 2ex);  
\node at (0,0)[vx, scale=.75](0){};
\node at (1,.5)[vx, scale=.75](1){};
\node at (2,0)[vx, scale=.75](2){};
\node at (.5,1)[vx, scale=.75](3){};
\node at (1.5,1.5)[vx, scale=.75](4){};
\node at (2.5,1)[vx, scale=.75](5){};
\node at (0,2)[vx, scale=.75](6){};
\node at (1,2.5)[vx, scale=.75](7){};
\node at (2,2)[vx, scale=.75](8){};
\node at (3.25,1)[vx, scale=.75](9){};
\draw (0)--(1)--(2)--(0);
\draw (3)--(4)--(5)--(3);
\draw (6)--(7)--(8)--(6);
\draw (0)--(3)--(6)--(0);
\draw (1)--(4)--(7)--(1);
\draw (2)--(5)--(8)--(2);
\end{tikzpicture} 
\\ \cline{2-2}\\[-2ex]
&\begin{tikzpicture}[xscale=.75, yscale=.5]
\node at (0,0)[vx, scale=.75](0){};
\node at (1,.5)[vx, scale=.75](1){};
\node at (2,0)[vx, scale=.75](2){};
\node at (.5,1)[vx, scale=.75](3){};
\node at (1.5,1.5)[vx, scale=.75](4){};
\node at (2.5,1)[vx, scale=.75](5){};
\node at (0,2)[vx, scale=.75](6){};
\node at (1,2.5)[vx, scale=.75](7){};
\node at (2,2)[vx, scale=.75](8){};
\node at (3.25,1)[vx, scale=.75](9){};
\draw (0)--(1)--(2)--(0);
\draw (3)--(4)--(5)--(3);
\draw (6)--(7)--(8)--(6);
\draw (0)--(3)--(6)--(0);
\draw (1)--(4)--(7)--(1);
\draw (2)--(5)--(8)--(2);
\foreach \x in {2,5,8} {\draw (9)--(\x);}
\end{tikzpicture} 
\\ \cline{2-2}\\[-2ex]
&\begin{tikzpicture}[xscale=.75, yscale=.5]
\node at (0,0)[vx, scale=.75](0){};
\node at (1,.5)[vx, scale=.75](1){};
\node at (2,0)[vx, scale=.75](2){};
\node at (.5,1)[vx, scale=.75](3){};
\node at (1.5,1.5)[vx, scale=.75](4){};
\node at (2.5,1)[vx, scale=.75](5){};
\node at (0,2)[vx, scale=.75](6){};
\node at (1,2.5)[vx, scale=.75](7){};
\node at (2,2)[vx, scale=.75](8){};
\node at (3.25,1)[vx, scale=.75](9){};
\draw (0)--(1)--(2)--(0);
\draw (3)--(4)--(5)--(3);
\draw (6)--(7)--(8)--(6);
\draw (0)--(3)--(6)--(0);
\draw (1)--(4)--(7)--(1);
\draw (2)--(5)--(8)--(2);
\foreach \x in {2,5,8} {\draw (9)--(\x);}
\draw (9) to [bend right=10](4);
\draw (9) to [bend left =25](3);
\end{tikzpicture} 
\\ \cline{2-2}\\[-2ex]
&\begin{tikzpicture}[xscale=.75, yscale=.5]
\node at (0,0)[vx, scale=.75](0){};
\node at (1,.5)[vx, scale=.75](1){};
\node at (2,0)[vx, scale=.75](2){};
\node at (.5,1)[vx, scale=.75](3){};
\node at (1.5,1.5)[vx, scale=.75](4){};
\node at (2.5,1)[vx, scale=.75](5){};
\node at (0,2)[vx, scale=.75](6){};
\node at (1,2.5)[vx, scale=.75](7){};
\node at (2,2)[vx, scale=.75](8){};
\node at (3.25,1)[vx, scale=.75](9){};
\draw (0)--(1)--(2)--(0);
\draw (3)--(4)--(5)--(3);
\draw (6)--(7)--(8)--(6);
\draw (0)--(3)--(6)--(0);
\draw (1)--(4)--(7)--(1);
\draw (2)--(5)--(8)--(2);
\foreach \x in {2,5,8} {\draw (9)--(\x);}
\draw (9) to [bend right=10](4);
\draw (9) to [bend left=15](1);
\draw (9) to [bend right=35](7);
\end{tikzpicture} 
\\ \cline{2-2}\\[-2ex]
&\begin{tikzpicture}[scale=.5]
\foreach \x in {0,1,2,3,4} {\node at (360/5*\x:1cm)[vx, scale=.75](a\x){};}
\foreach \x in {0,1,2,3,4} {\node at (360/5*\x+360/10:2cm)[vx, scale=.75](b\x){};}
\draw (a0)--(b4)--(a4)--(b3)--(a3)--(b2)--(a2)--(b1)--(a1)--(b0)--(a0);
\draw (0,0) circle (2cm);
\draw (0,0) circle (1cm);
\draw (a0)--(a2)--(a4)--(a1)--(a3)--(a0);
\end{tikzpicture}
\\
\end{tabular}
\tabularnewline\hline
\end{tabular}
\caption{Values of $\indsat{n}{\claw}$ for $4\leq n\leq10$ along with trigraphs realizing those values. All $\claw$-induced-saturated graphs for $n=9$ and $n=10$ are shown.}
\label{tab:claws}
\end{table}

\iftoggle{sub}{
\begin{thm}\label{thm:claw-sis}
For $n \ge 9$, $n \neq 14, 17,$ $2n - 2 \le \sis{n}{K_{1,3}} \le 2n + 2$.
\end{thm}

In order to prove Theorem \ref{thm:claw-sis}, we first prove a series of lemmas.
}{
\begin{thm}\label{thm:claw-sis}
The following bounds hold for $n\geq9$, $n\neq14,17$:
\begin{center}
\begin{tabular}{l c l}
\ $\sis{n}{\claw}=2n$ & if & $n\equiv0\mod3$\\
\ $\sis{n}{\claw}=2n-2$ & if & $n\equiv1\mod3$\\
\ $2n\leq\sis{n}{\claw}\leq2n+2$ & if & $n\equiv2\mod3$
\end{tabular}
\end{center}
\end{thm}

In order to prove Theorem~\ref{thm:claw-sis}, we first prove a series of lemmas that will aid in producing the lower bounds of the statement.
Then we construct families of \claw-induced-saturated graphs that exhibit the upper bounds of Theorem~\ref{thm:claw-sis}.

The following lemma shows that \claw-induced-saturated graphs have few vertices of low degree.
}

\begin{lem}\label{lem:sis_clawdeg}\label{deg2}
Let $G$ be a \claw-induced-saturated graph. Then $G$ has
\begin{enumerate}
\item at most one isolated vertex,
\item no vertices of degree one,
\item at most one vertex of degree two, and
\item at most two vertices of degree three.
\end{enumerate}
Furthermore, if $G$ has an isolated vertex $v$, then $\delta(G-v)\geq4$.
\iftoggle{sub}{}{
Additionally, if $G$ has a vertex of degree three, then $G$ does not have a vertex of degree two. 
If $G$ has two vertices of degree three or a vertex of degree two, then $G$ has a vertex of degree at least five.
}
\end{lem}
\begin{proof}
Let $G$ be a \claw-induced-saturated graph.  Observe that if we had two isolated vertices, then adding the edge between them would not yield a $K_{1,3}$. Also, any edge of $G$ lies in a triangle, so there are no vertices of degree one. 

Suppose that $u$ and $v$ are vertices of degree two. Since every edge lies in a triangle the neighbors of $u$ are adjacent, as are the neighbors of $v$.  Thus, if $u$ and $v$ are not adjacent, adding the edge $uv$ does not create an induced \claw.  If $u$ and $v$ are adjacent, then $N[u] = N[v] = \{u,v,w\}$ for some $w$.  However, removing $uw$ does not create an induced \claw as $v$ would have to have been its center.  So $G$ has at most one vertex of degree two.

To prove $(4),$ suppose $u$ is a vertex of degree three with neighbors $u_1,u_2,u_3$.  Since every edge is in a triangle, we may assume that $u_1u_2, u_2u_3 \in E(G)$.  
\textit{Case 1:} $u_1u_3 \notin E(G)$. Then adding $u_1u_3$ creates an induced \claw centered at either $u_1$ or $u_3$; say $u_1$.  Then $u_1$ has two nonadjacent neighbors $x$ and $y$ that are distinct from $u_2$ and $u_3$.  However, $\{u,u_1,x,y\}$ induces a \claw in $G$, a contradiction.  \textit{Case 2:} $u_1u_3 \in E(G)$. In particular, every vertex of degree three in $G$ is contained in a $K_4$.  Let $v$ be another vertex of degree three.  By the above, $N[v]$ induces $K_4$.  If $uv \notin E(G)$, then adding $uv$ does not create an induced \claw.  Thus, $u$ and $v$ are adjacent, and consequently the only vertices of degree three are contained in $N[u]$.  

If we remove $uu_1$, then an induced \claw exists, centered at either $u_2$ or $u_3$.   So at least one of them has degree at least four, say $u_3$.  Similarly, removing $uu_3$ creates an induced \claw centered at either $u_1$ or $u_2$ so that at least one of them has degree at least four.  In any case, at most two vertices in $N[u]$, and as a result in $G$, have degree three. Thus, $(4)$ holds.

If $G$ has an isolate, $u$, and another vertex $v$ with $\deg(v) \le 2$, then adding $uv$ cannot create an induced \claw unless $\deg(v) = 2$.  In this case, the neighbors of $v$ cannot be adjacent, however every edge of $G$ must be in a triangle, a contradiction.

\iftoggle{sub}{}{
Suppose $u$ and $v$ are vertices with $\deg(u) = 2$ and $\deg(v) = 3$.  By previous arguments, the neighbors of $u$ form a clique, as do the neighbors of $v$.  Thus, if $uv \notin E(G)$, adding $uv$ does not create an induced $\claw$.  So $uv \in E(G)$, and in particular, $u$ is in the $K_4$ induced by $N[v]$.  However, $\deg(u) = 2$, a contradiction.

Now, suppose $u$ and $v$ are vertices with $\deg(u)=\deg(v)=3$.  By the above, they must be contained in the same $K_4$, so let $u,v,x,y$ denote the vertices of this $K_4$.
If we delete $xy$, then $x$ and $y$ are the leaves of a \claw, but this \claw is not centered at $u$ or $v$, so $x$ and $y$ have a common neighbor $z \not\in \{u,v\}$. 
If we delete $xz$, the resulting \claw is centered at a common neighbor of $x$ and $z$. 
If that common neighbor is not $y$, then $\deg(x) \geq 5$, and if it is, then $\deg(y) \ge 5$. 

Similarly, suppose $\deg(v)=2$, with $N(v)=\{x,y\}$. 
Since every edge is in a triangle, $xy \in E(G)$. 
If we consider deleting the edge $xy$, we note that the \claw formed does not have center $v$, so $x$ and $y$ share another neighbor $z$, and $z$ has a neighbor $z'$ nonadjacent to both $x$ and $y$. 
Consider deleting the edge $vx$. The \claw formed must be centered at $y$, so $y$ has a neighbor nonadjacent to  $v$ or $x$. 
Then this neighbor $y'$ is not any of the vertices already named. Similarly, $x$ has a neighbor $x' \not\in \{v,x,y,y',z,z'\}$. 
Then $\{x',y'\} \subseteq N(z)$ else $G[x,x',v,z]$ or $G[y,y',v,z]$ is a \claw.
Thus, $\deg(z)\geq 5$.}
\end{proof}

\iftoggle{sub}{}{
\begin{cor}\label{cor:sis_clawmin}
Any graph that is \claw-induced-saturated (on $n \geq 9$ vertices) has at least $2n-2$ edges.
That is, $\sis{n}{\claw} \geq 2n-2$ for $n \geq 9$. 
Furthermore, if $G$ is a \claw-induced-saturated graph that does not have an isolated vertex, then $e(G)\geq 2n$.
\end{cor}
\begin{proof}
Apply the degree-sum formula and Lemma~\ref{lem:sis_clawdeg}.
\end{proof}


As indicated in Corollary~\ref{cor:sis_clawmin}, if a graph on $n$ vertices obtaining the minumum number of edges among \claw-induced-saturated graphs exists, then it is four-regular except for an isolated vertex.
We provide the following structural results to show such a graph only exists if $n\equiv1\mod 3$.

\begin{lem}\label{lem:clawnbrs}
Suppose $G$ is a \claw-induced-saturated graph, and for some $v \in V(G)$, every vertex in $N[v]$ has degree precisely 4.
Then $G[N(v)] \in \{2K_2,P_4\}$.
\end{lem}

\begin{proof} 

Since we are assuming every vertex in $N[v]$ has degree 4, then we can let $N(v)=\{u,x_1,x_2,x_3\}$. Next, we show that $\Delta(G[N(v)])\leq 2$. Suppose to the contrary that some vertex, say $u\in N(v)$, has three neighbors within $N(v)$; hence, $N(u)\cap N(v)=\{x_1, x_2,x_3\}$. By deleting $ux_1$, we see that $u$ and $x_1$ have a common neighbor besides $v$. 
Using the symmetry of $x_1,x_2$, and $x_3$, without loss of generality $x_1x_2,x_2x_3 \in E(G)$. 
Now $N(x_2)=\{u,v,x_1,x_3\}$, because $\deg(x_2)=4$.
Consider deleting $ux_1$. 
The common neighbors of $u$ and $x_1$ are $v,x_2,$ and maybe $x_3$. 
Neither $v$ nor $x_2$ can be the center of a \claw since all of their neighbors are adjacent to $u$ or $x_1$.
Hence $x_3$ must be the center of the induced \claw, so $x_1x_3 \in E(G)$.  But then $N(x_3)=\{v,u,x_1,x_2\}$ so the \claw supposedly centered at $x_3$ has no third leaf.

This shows that  $\Delta(G[N(v)]) \leq 2$. 
Because every edge is in a triangle, if $\Delta(G[N(v)]) < 2$, then $G[N(v)]=2K_2$, so suppose $\Delta(G[N(v)])=2$. 
Then $G[N(v)]$ is either $C_4$ or $P_4$.

Suppose $x_1x_2x_3x_4=P_4 \subseteq G[N(v)]$. 
If $x_1x_4 \not\in E(G)$, then $G[N(v)]=P_4$, so suppose $x_1x_4 \in E(G)$. 
Deleting the edge $x_2x_3$ shows that $x_2$ and $x_3$ have a common neighbor $y \in V(G) \setminus N[v]$. Separately, consider deleting $x_3x_4$. The only possible common neighbors of $x_3$ and $x_4$ are $v$ and $y$. Because $x_1x_4 \in E(G),$  $v$ cannot be the center of the \claw created by deleting $x_3x_4$, so the center is $y$. Then the third leaf must be some vertex $y' \not\in N(x_3) \cup N(x_4)$. But we also know that $y' \not\in N(x_2)$, since $\deg(x_2)=4$, so $G[y,y',x_2,x_4]$ is an induced \claw, a contradiction.
\end{proof}

For the remainder of this section, we define $R(G):=\{v\in V(G):G[N(v)]=2K_2\}$ and $B(G):=\{v\in V(G):G[N(v)]=P_4\}$ for any graph $G$. 
Hence if $G$ is a four-regular \claw-induced-saturated graph, then $V(G)$ is partitioned into $R(G)$ and $B(G)$.  
We will call the vertices in $R(G)$ {\it red vertices} and those in $B(G)$ {\it blue vertices}.


\begin{lem}\label{lem:bluetri}
If $G$ is a 4-regular \claw-induced-saturated graph, then $B(G)$ induces $kK_3$ for some $k$.
\end{lem}
\begin{proof}
Let $v \in B(G)$ so that $G[N(v)]$ is a path $x_1 x_2 x_3 x_4$.  Since $P_3\subseteq G[\{v,x_1,x_3\}]\subseteq G[N(x_2)]$ and $P_3\subseteq G[\{v,x_2,x_4\}]\subseteq G[N(x_3)]$, Lemma~\ref{lem:clawnbrs} implies that $x_2,x_3\in B(G)$.  
Furthermore, as deleting $x_2x_3$ creates an induced \claw, which cannot be centered at $v$, then $x_2$ and $x_3$ share another common neighbor, call it $y$. Since $N(x_2) = \{v,x_1,x_3,y\}$,  $x_1 \in B(G)$ iff $x_1$ and $y$ are neighbors.  So if $x_1y \in E(G)$, we consider adding $vy$ to $G$.  This creates an induced \claw, which must be centered at $y$.  However, since $G$ is 4-regular, $y$ has at most one neighbor outside of $\{x_2,x_3,x_4,v\}$ and cannot be the center of this induced \claw, a contradiction.  Thus, $x_1 \in R(G)$, and by symmetry, $x_4 \in R(G)$.  Repeating the above argument for $x_2$ instead of $v$ shows that $y \in R(G)$.  Hence, $\{x_2,x_3, v\}\subseteq B(G)$ but $N(\{x_2,x_3, v\})=\{x_1,x_4,y\}\subseteq R(G)$ and so $\{x_2,x_3, v\}$ induces a traingle of vertices in $B(G)$. 
\end{proof}

An example of a 4-regular \claw-induced-saturated graph, with $R(G),B(G)\neq\emptyset$ is shown in Figure~\ref{fig:claw pretty}.  Observe that $B(G)$ induces $8K_3$, which is in accordance with Lemma~\ref{lem:bluetri}.

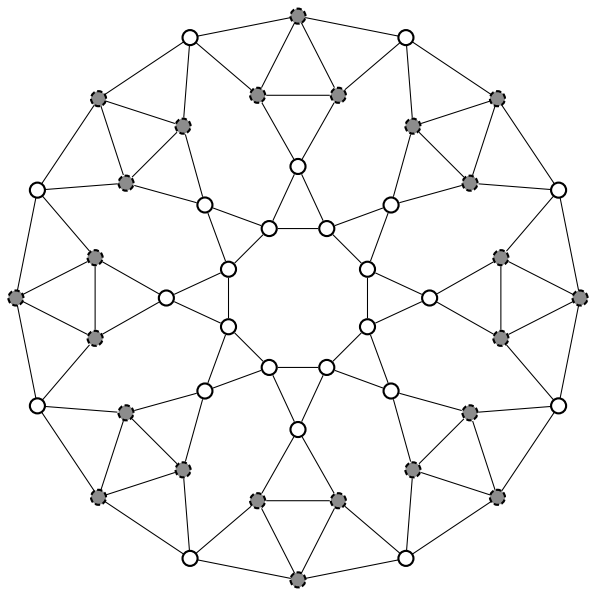
\begin{figure}[ht]
\begin{tikzpicture}
\foreach \x in {0,...,7}
{
\draw[red] (0,0) + (90-360/8*\x:1.75cm) node[whitevx] (A\x){};
\draw[red] (0,0) + (90-360/8*\x+360/16:1cm) node[whitevx] (B\x){};
}
\foreach \x in {0,...,15}
{
\draw[blue] (0,0) + (90-360/16*\x+360/16-360/32:2.75cm) node[grayvx] (C\x) {};
\draw (0,0) + (90-360/16*\x+360/16:3.75cm) node[inner sep=0] (D\x) {};
}
\foreach \x in {0,...,6}
{
\pgfmathparse{int(\x+1)}
\draw (B\x) -- (B\pgfmathresult);
\draw (B\x) -- (A\x) -- (B\pgfmathresult);
}
\draw (B0) -- (A7) -- (B7) -- (B0);
\foreach \x in {0,...,14}
{
\pgfmathparse{int(\x+1)}
\draw (D\x) -- (D\pgfmathresult);
\draw (D\x) -- (C\x) -- (D\pgfmathresult);
}
\draw (D15) -- (C15) -- (D0) -- (D15);
\foreach \x in {0,2,4,6,8,10,12,14}
{
\pgfmathparse{int(\x+1)}
\draw (C\x) -- (C\pgfmathresult);
\draw[red] (D\x)node[whitevx] {};
\draw[blue] (D\pgfmathresult)node[grayvx] {};
}
\foreach \x in {0,...,7}
{
\pgfmathparse{int(2*\x)}
\draw (A\x) -- (C\pgfmathresult);
\pgfmathparse{int(2*\x+1)}
\draw (A\x) -- (C\pgfmathresult);
}
\end{tikzpicture}\caption{A 4-regular \claw-induced-saturated graph. Vertices in $R(G)$ are white, and vertices in $B(G)$ are gray.}\label{fig:claw pretty}
\end{figure}

\begin{lem}\label{lem:adj}
Let $G$ be a 4-regular \claw-induced-saturated graph.  
Every edge of $G$ is in either one or two triangles, and there are $|B(G)|$ edges that are in two triangles.
\end{lem}
\begin{proof}
Recall that every edge in a \claw-induced-saturated graph is in at least one triangle.
Suppose there exists $xy \in E(G)$ where $x$ and $y$ have three common neighbors $u,v,w$.
Then $G[N(x)]$ cannot be in $\{2K_2, P_4\}$, which contradicts Lemma \ref{lem:clawnbrs}. Hence, each edge is in at most two triangles.

Let $b$ be the number of edges that are in two triangles. 
Label edge $xy$ with vertex $z$ if $xyz$ is a triangle, and allow for multiple labels.  
Thus $b$ edges have two labels and hence
$$|E(G)|+b=\sum_{z\in V(G)}|\{e\in E(G): e \text{~has~label~} z\}|.$$
Since each red vertex gives its label to two triangles and each blue vertex gives its label to three triangles, we have $\sum_{z\in V(G)}|\{e\in E(G): e \text{~has~label~} z\}|=2|R(G)|+3|B(G)|$. Thus, since $G$ is 4-regular,

\begin{align*}
2n+b&=|E(G)|+b\\
&= \sum_{z\in V(G)}|\{e\in E(G): e \text{~has~label~} z\}|\\
&=2|R(G)|+3|B(G)|\\
&= 2(n-|B(G)|)+3|B(G)|\\
&=2n+|B(G)|\\
\end{align*}

Therefore, there are precisely $|B(G)|$ edges that are in two triangles. 
\end{proof}


\begin{prop}\label{3|n}
If $G$ is a 4-regular, \claw-induced-saturated graph on $n$ vertices, then $n\equiv0\mod3$.
\end{prop}
\begin{proof}
Let $b=|B(G)|$.
By Lemma~\ref{lem:bluetri}, 3 divides $b$.
By Lemma~\ref{lem:adj}, $2n-b$ edges are in precisely one triangle, and $b$ edges are in precisely two triangles. 
If $t$ is the number of triangles in $G$, then
$3t=(2n-b)+(2b)=2n+b$. Since 3 divides $b$, we know 3 divides $2n$, and so 3 divides $n$.
\end{proof}

The previous lemmas will be used in the proof of Theorem~\ref{thm:claw-sis} to obtain a lower bound $\sis{n}{\claw}\geq2n-1$ for $n \equiv 0 \mod 3$. 
The next two propositions show that certain degree sequences do not have a \claw-induced-saturated realization.
This allows us to increase the lower bound of $\sis{n}{\claw}$ for certain values of $n$.

\begin{prop}\label{554}If $G$ is a \claw-induced-saturated graph, then the degree sequence of $G$ is not $(5,5,4,\ldots,4)$.
\end{prop}
\begin{proof}
Suppose $G$ is a counterexample to the claim, and let $v$ be a vertex of degree 5.

\begin{case} $\Delta(G[N(v)])=4$.\end{case} 
That is, $v$ has a neighbor $u$ such that $X:=N(u)\cap N(v)$is a set of order 4. 
If we delete $vx'$ for some $x' \in X$, then the resulting \claw is not centered at $u$ since the neighbors of $u$ are adjacent to $v$.
Thus $x'$ and $v$ share a neighbor $x\in X$ and there is some $\in N(x)\setminus[N(x')\cup N(v)]$. 
Now $N(x)=\{u,v,x',y\}$ and $uy,vy,x'y\notin E(G)$, so the edge $xy$ is in no triangle, a contradiction.

\begin{case}\label{Case3} $\Delta(G[N(v)])=3$. \end{case}
That is, there exist $u \in N(v), w\notin N[u],$ and $X\subseteq N(u)$ with $|X|=3$ so that  $N(v)=\{u,w\} \cup X$. Since deleting the edge $vw$ creates a \claw, 
there exist vertices $x'$ and $y$ such that $x'$ is a common neighbor of $v$ and $w$, $y$ is adjacent to $x'$, and $y$ is not adjacent to $w$ nor $v$. Note $x' \in X$ and $y \notin N[v]$.
Then to prevent a \claw in $G$ with center $x'$ and leaves $u,w,y$, we have $uy \in E(G)$. Then, $u,v$ are the vertices of degree 5 and all other vertices have degree 4 so that $N(x')=\{u,v,y,w\}$. Since $u$ is not the center of a \claw, and $x'$ has no neighbors in $X$, the vertices of $X\setminus\{x\}$ (call them $a$ and $b$) are adjacent. 

Note that $u$ was chosen as an arbitrary vertex of $N(v)$ with three neighbors in $N(v)$, and we showed $\deg(u)=5$. 

Now, $\deg(a)=\deg(b)=4$ and each of $a$ and $b$ currently has two neighbors in $N(v)$. If $a$ (or $b$) were adjacent to $w$, then the argument previously applied to $u$ would guarantee that $\deg(a)=5$ (or $\deg(b)=5$), thus giving us at least three vertices of degree 5. Therefore $a$ and $b$ each have a neighbor outside of $N[v]$; due to the necessity that every edge be in a triangle, they share this neighbor, which we shall name $z$. 
It is possible that $z=y$.  Suppose $z\neq y$, then deleting $az$ should create an induced \claw centered at at common neighbor of $a$ and $z$.  However, the only option is $b$, which is not the center of such a \claw, a contradiction.  So suppose $z = y$, then $\deg_{N(u)}(a)=3$. By the previous argument with $u$, we must have $\deg_G(a)=5$, a contradiction.

\begin{case} $\Delta(G[N(v)])\leq 2$. \end{case}
$N(v)$ has no independent set of size three, lest it be the center of a \claw.
Then $G[N(v)] \in \{K_2 + K_3,C_5\}$. 

Suppose first $G[N(v)]=K_2 + K_3$, with $\{x_1,x_2,x_3\}$ inducing $K_3$. 
We may suppose $\deg(x_1)=\deg(x_2)=4$ since at most one of the vertices in the copy of $K_3$ may have degree 5. So each of $x_1$ and $x_2$ have a neighbor outside of $N[v]$, say $y$ and $z$, respectively.  If $y \neq z$, then since every edge is contained in a triangle, $x_3$ is adjacent to both $y$ and $z$.  However, this implies that $\deg(z) = 5$ and $\Delta(G[N(z)]) \ge 3$, as evidenced by $x_1$.  This puts us in Case \ref{Case3}.

So $y = z$, and consequently, $N[x_1] = N[x_2]$.   The only common neighbors of $x_1$ and $y$ are $x_2$ and possibly $x_3$. If $x_3\notin N(x)\cap N(y)$, removing $x_1y$ should create an induced \claw centered at $x_2$, but it does not.  Thus, removing $x_1y$ creates an induced \claw centered at $x_3$, which implies that $x_3$ is adjacent to $y$, as well as another vertex $y'$ not in $N[v] \cup \{y\}$.  However, this implies that $\deg(x_3)=5$, and $\Delta(G[N(x_3)]) \geq 3$, as evidenced by $x_1$.  This also puts us in Case \ref{Case3}.

Suppose now $G[N(v)]=C_5$ with cycle $x_1x_2x_3x_4x_5$. 
We may assume that $x_1,x_2,x_3,$ and $x_4$ all have degree 4.
The only common neighbors of $v$ and $x_2$ are $x_1$ and $x_3$.  When removing $vx_2$ we obtain an induced claw centered at either $x_1$ or $x_3$.  Without loss of generality, assume it is $x_3$.  Since $x_4$ is not a leaf of a \claw that  features $v$ as a leaf, $x_3$ has a neighbor $y \not\in N(v) \cup N(x_2)$. Since $G[x_3,x_2,x_4,y]$ cannot be a \claw, we must have $yx_4 \in E(G)$. Similarly, if we delete $vx_3$, the candidates for center of the ensuing \claw are $x_2$ and $x_4$; we know the neighborhood of $x_4$, and so see that $x_2$ is the center. Then there exists $y' \in N(x_2)$ such that $y' \not\in N(v) \cup N(x_3)$, and as before $y'x_1 \in E(G)$. Now we know the neighborhoods of $x_1,x_2,x_3$, and $x_4$. If we add the edge $x_1x_4$, we find that no \claw is formed, a contradiction. 
\end{proof}

\begin{prop}\label{644}
Let $G$ be a \claw-induced-saturated graph. Then for any $n \geq 7$, the degree sequence of $G$ is \emph{not} $(6,4,\ldots,4)$.
\end{prop}
\begin{proof}
Suppose $G$ is a counterexample to this claim. Let $v$ have degree six, and let $F=G[N(v)]$, so $|F|=6$, $\Delta(F) \leq 3$, and $\alpha(F) \leq 2$ else $v$ is the center of a \claw. In fact $\alpha(F)=2$ in order for the vertices of $N(v)$ to have degree four in $G$. 
If $\delta(F)=3$, then $N[v]$ is a component of $G$, and this component is \claw-induced-saturated. 
However, from the computer search, with results listed in Table~\ref*{tab:claws}, we know that there is no nontrivial \claw-induced-saturated graph on fewer than nine vertices. 
Therefore $\delta(F) \leq 2$. 
Indeed, we claim $\delta(F)=2$. 
If $\delta(F) \leq 1$, let $u$ be a vertex with minimum $F$-degree (i.e. $\deg_F(u)$ is minimum), and let $T=F\setminus N[u]$. 
Then $T$ is a clique, else two nonadjacent vertices in $T$ together with $u$ and $v$ form a \claw.
Hence $|T|=4$ and the vertices of $T$ have no neighbors outside of $N[v]$ in $G$. 
Now, deleting the edge between $v$ and any vertex of $T$ does not create an induced \claw, so $\delta(F)=2$.

Let $u$ be a vertex in $F$ with $\deg_F(u)=2$, and let $T=F\setminus N[u]$. 
As before, $T$ is a clique, specifically a triangle.   Let $N_F(u) = \{u', u''\}$. 
Since $\deg_F(u)=2$, $u$ has one neighbor $w$ outside of $N[v]$; since every edge is in a triangle, we may assume that $w$ is adjacent to $u'$.  Now, the only common neighbors of $v$ and $u$ are in $N_F(u)$. 
Since $\delta(F)=2$, $u'$ must have another neighbor in $F$ other than $u$.  Thus, the only neighbor of $u'$ not in $N[v]$ is $w$, and if we delete $vu$, the resulting induced \claw cannot be centered at $u'$.  So it must be centered at $u''$, which in turn has a neighbor $w''$ outside of $N[v] \cup \{w\}$.  Since $u''w''$ is in a triangle and $\delta(F) = 2$, $u''$ and $w''$ share a neighbor $t''$ in $F$. 
Since $\delta(F) = 2$, no vertex in $F$ has two neighbors outside $N[v]$.  So $t'' \neq u'$, and hence $t'' \in T$. 
But now $\deg(t'') \geq 5$, a contradiction.
\end{proof}

Finally, we construct graphs which we use to find an upper bound for $\sis{n}{\claw}$.
}

\begin{lemma}\label{suff_claw}
If $G$ is a graph where the neighborhood of every vertex induces $2K_2$, then $G$ is \claw-induced-saturated.
\end{lemma}
\begin{proof}
Since no vertex has three independent neighbors, $G$ contains no induced \claw.   Suppose we delete an edge $xy$.  
Since every edge is in a triangle, say $xyz$, deleting $xy$ leaves $z$ as the center of a \claw with leaves $x$, $y$, and any other neighbor of $z$.
If we add an edge between two vertices with no common neighbors, then we take the new edge together with two nonadjacent neighbors of one of the vertices and find a \claw.
Therefore it suffices to consider adding an edge $xy$, where $x$ and $y$ share a neighbor. 
Let $N(x)=\{u_1,u_2,v_1,v_2\}$ with $u_1u_2,v_1v_2 \in E(G)$, and suppose $u_1 \in N(y)$.
Then $u_2 \notin N(y)$ otherwise $N(u_2)$ would contain a $P_3$ and not be $2K_2$.  Similarly, both $v_1$ and $v_2$ cannot be in $N(y)$.  So we may assume $v_2 \notin N(y)$.  Then upon adding $xy$, $\{x,y,u_2,v_2\}$ induces a \claw.
\end{proof}

\begin{lemma}
Let $G$ be a graph with at most one isolated vertex, where each nontrivial  component is one of the graphs in Figure~\ref{fig:claw-i-s}.  Then $G$ is \claw-induced-saturated.
\end{lemma}

\begin{figure}
\centering
\subcaptionbox{$H=K_3\cartprod K_3$, 9 vertices\label{k3xk3}}[.4\linewidth]{
\begin{tikzpicture}[xscale=2, yscale=1.5]
\node at (0,0)[vx](0){};
\node at (1,.5)[vx](1){};
\node at (2,0)[vx](2){};
\node at (.5,1)[vx](3){};
\node at (1.5,1.5)[vx](4){};
\node at (2.5,1)[vx](5){};
\node at (0,2)[vx](6){};
\node at (1,2.5)[vx](7){};
\node at (2,2)[vx](8){};
\draw (0)--(1)--(2)--(0);
\draw (3)--(4)--(5)--(3);
\draw (6)--(7)--(8)--(6);
\draw (0)--(3)--(6)--(0);
\draw (1)--(4)--(7)--(1);
\draw (2)--(5)--(8)--(2);
\end{tikzpicture}
}\hspace{1cm}
\subcaptionbox{Graph $J$ on 11 vertices\label{11vtcs}}[.4\linewidth]{
\begin{tikzpicture}[scale=.75]
\draw (0,0) node[vx](b1) {};
\draw (2,0) node[vx](b2) {};
\draw (1,1) node[vx](b3) {};
\draw (b1)--(b2)--(b3)--(b1);
\draw (-1,2) node[vx](r1) {};
\draw (3,2) node[vx](r2) {};
\draw (b1)--(r1)--(b3)--(r2)--(b2);
\draw (1,3) node[vx](g1) {};
\draw (0,4) node[vx](g2) {};
\draw (2,4) node[vx](g3) {};
\draw (1,5) node[vx](g4) {};
\foreach \x in {1,...,4}
	{\foreach \y in {1,...,\x}
		\draw (g\x)--(g\y);}
\draw (-1,6) node[vx](b4) {};
\draw (3,6) node[vx](b5) {};
\draw (g2)--(b4)--(g4)--(b5)--(g3);
\draw (g2)--(r1)--(b4)--(b5)--(r2)--(g3);
\draw (b1)--(g1)--(b2);
\end{tikzpicture}
}
\subcaptionbox{Graph $K$ on 12 vertices\label{12vtcs}}[.4\linewidth]{
\begin{tikzpicture}[scale=0.25]
\draw (0,0) circle (9cm);
\foreach \x in {1,...,6}
{
\draw (150-60*\x:7cm) node[vx] (A\x) {};
}
\foreach \x in {1,3,5}
{
\draw (120-60*\x:3cm) node[vx] (B\x) {};
}
\foreach \x in {2,4,6}
{
\draw (120-60*\x:9cm) node[vx] (C\x) {};
\pgfmathparse{int(\x-1)}
\draw (C\x) -- (A\x) -- (A\pgfmathresult);
}
\draw (A1) -- (C6);
\draw (A3) -- (C2);
\draw (A5) -- (C4);
\foreach \x in {1,3,5}
{
\pgfmathparse{int(\x+1)}
\draw (B\x) -- (A\x);
\draw (A\pgfmathresult) -- (B\x);
}
\draw (B1) -- (B3) -- (B5) -- (B1);
\draw (A1) -- (A6);
\draw (A2) -- (A3);
\draw (A4) -- (A5);
\end{tikzpicture}
}
\subcaptionbox{Graph $L$ on 15 vertices\label{15vtcs}}[.4\linewidth]{
\begin{tikzpicture}
\draw (0,0) circle (2.25cm);
\foreach \x in {1,2,3,4,5}{\draw  (360/5*\x:0.75cm)  node[vx](a\x){};}
\foreach \x in {1,2,3,4,5}{\draw  (36+360/5*\x:1.5cm)  node[vx](b\x){};}
\draw (a1)--(a2)--(a3)--(a4)--(a5)--(a1);
\foreach \x in {1,2,3,4,5}{\draw  (360/5*\x:2.25cm)  node[vx](c\x){};}
\draw (a1)--(b1)--(a2)--(b2)--(a3)--(b3)--(a4)--(b4)--(a5)--(b5)--(a1);
\draw (c1)--(b1)--(c2)--(b2)--(c3)--(b3)--(c4)--(b4)--(c5)--(b5)--(c1);
\end{tikzpicture}
}
\caption{These graphs are \claw-induced-saturated.}\label{fig:claw-i-s}
\end{figure}
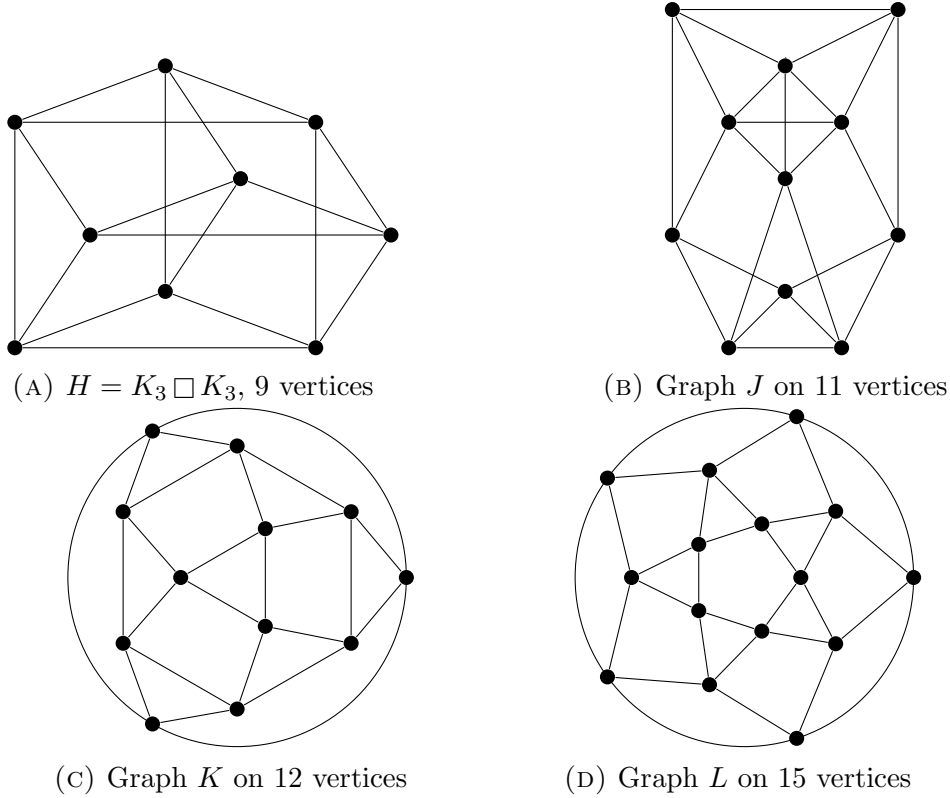

\begin{proof}
By inspection, the graph in Figure ~\ref{11vtcs} is \claw-induced-saturated, and since the graphs in Figures \ref{k3xk3}, \ref{12vtcs}, and \ref{15vtcs} have the property that the neighborhood of every vertex induces $2K_2$, they are \claw-induced-saturated by Lemma~\ref{suff_claw}.  

Now let $G$ be a graph with at most one isolated vertex and each of the remaining components are one of the graphs from Figure~\ref{fig:claw-i-s}.
Since each nontrivial component of $G$ is \claw-induced-saturated, we only need to consider adding an edge $xy$ between components.    When we add the edge $xy$, at least one of $x$ and $y$ must be in a nontrivial component, say $x$.  By inspection we see every vertex in every graph of Figure \ref{fig:claw-i-s} has two nonadjacent neighbors, and in particular, this holds for $x$.  Thus, $x$ together with these two neighbors  and $y$ induce a \claw.  Therefore, $G$ is \claw-induced-saturated.
\end{proof}


We now can prove Theorem \ref{thm:claw-sis}.

\begin{proof}[Proof of Theorem \ref{thm:claw-sis}]
\iftoggle{sub}{
Lemma \ref{lem:sis_clawdeg} together with the degree-sum formula gives us the lower bound of $2n - 2$.  To obtain the upper bound, we contruct graphs using $H, J, K$, and $L$ from Figure \ref{fig:claw-i-s}.  These constructions will be based on the residue class of $n$ modulo 3.   

\begin{case}\label{0mod3}
$n \equiv 0 \mod 3$, $n \geq 9$
\end{case}
Use $\lfloor n/9 \rfloor-1$ copies of $H$, together with one copy of $H$, $K$, or $L$, for a graph with $2n$ edges.

\begin{case}
$n \equiv 1 \mod 3$, $n \geq 10$
\end{case} 
Use an isolated vertex with a graph from Case~\ref{0mod3} for a graph with $2n-2$ edges.
\begin{case}
$n \equiv 2 \mod 3$, $n\geq20$ or $n=11$.
\end{case}
If $n=11$, the graph $J$ suffices. If $n > 11$, then take $J$ and a construction from Case~\ref{0mod3}. This achieves $2n+2$ edges.
}
{
We exhibit graphs with the desired number of edges to prove the upper bounds.

\begin{case}\label{0mod3}
$n \equiv 0 \mod 3$, $n \geq 9$
\end{case}
Use $\lfloor n/9 \rfloor-1$ copies of $H$, together with one copy of $H$, $K$, or $L$, for a graph with $2n$ edges.
Alternatively, we could generalize $L$ for $n\geq 15$ by having $n/3$ vertices in the outer cycle, $n/3$ vertices in the inner cycle, and $n/3$ vertices between the two cycles.
\begin{case}
$n \equiv 1 \mod 3$, $n \geq 10$
\end{case} 
Use an isolated vertex with a graph from Case~\ref{0mod3} for a graph with $2n-2$ edges.
\begin{case}
$n \equiv 2 \mod 3$, $n\geq20$ or $n=11$.
\end{case}
If $n=11$, the graph $J$ suffices. If $n \geq 20$, then take $J$ and a construction from Case~\ref{0mod3}. This achieves $2n+2$ edges.

For the lower bound, let $G$ be any \claw-induced-saturated graph.
Corollary \ref{cor:sis_clawmin} gives us a general lower bound of $2n-2$. 
Suppose $G$ has no isolated vertex. Then by Corollary~\ref{cor:sis_clawmin}, $e(G) \geq 2n$, as desired. Suppose then that $G$ does have an isolated vertex, and $n \not\equiv 1 \mod 3$. Then $(n-1) \not\equiv 0 \mod 3$, so by Lemmas~\ref{lem:sis_clawdeg} and \ref{3|n}, the minimum degree of the non-isolated vertices is at least 4, and $\Delta(G) \geq 5$. Then 
$e(G) \geq \big\lceil\frac{4(n-1)+1}{2}\big\rceil=2n-1$, with equality only if the degree sequence of $G$ is $(5,5,4,\ldots,4,0)$ or $(6,4,\ldots,4,0)$. Since the graph obtained by deleting the isolate is \claw-induced-saturated, by Propositions~\ref{554} and $\ref{644}$, $e(G) \geq 2n$.}
\end{proof}

\iftoggle{sub}{
It is worth noting that for $n \equiv 1 \mod 3$, $n \ge 10$, $\sis{n}{K_{1,3}} = 2n - 2$.  Additionally, a detailed case analysis of degree sequences of $K_{1,3}$-induced-saturated graphs  shows that $\sis{n}{K_{1,3}} \ge  2n$ for $n \not\equiv 1 \mod 3$.  In particular, this implies $\sis{n}{K_{1,3}} = 2n$ for $n \equiv 0 \mod 3$, $n \ge 9$, and $2n \le \sis{n}{K_{1,3}} \le 2n + 2$ for $n \equiv 2 \mod 3$, $n \ge 20$.
}{}


%
%

\section{$C_4$ and its complement}\label{sec:small cycles}

In this section we show that the induced saturation number of $C_4$ is zero for sufficiently large $n$, and we compute some bounds on $\sis{n}{C_4}$. Additionally, using Observation \ref{thm:complement} and the fact that $\overline{C_4}=2K_2$, we use $C_4$-induced-saturated graphs to obtain results for matchings. 


\begin{cons}\label{gen gen icos}
For $j \geq 5$ and $k \geq 2$, let $I_j^k$ be the graph that combines $k$ copies of a wheel with $j$ spokes.  Label the copies $W^1,\ldots,W^k$, and label the vertices of $W^i$ so that its center is $w_0^i$, and the outer cycle of $W^i$ is $w_1^i,\ldots,w_j^i$. For $1\leq i<i^*$,  add the edges $w^i_\ell w^{i^*}_\ell$ and $w^i_\ell w^{i^*}_{\ell+1}$ for every $\ell \in [j]$, defining $j+1:=1$. \end{cons}

$I_5^2$ is the icosahedron, shown in Figure \ref{fig:icos}. The icosahedron can be thought of as two wheels with 5 spokes whose outer-cycle vertices are joined by a zig-zag pattern (as described precisely in Construction \ref{gen gen icos}). Construction~\ref{gen gen icos} generalizes the icosahedron by allowing the number of wheels and the length of their outer cycles to vary. 

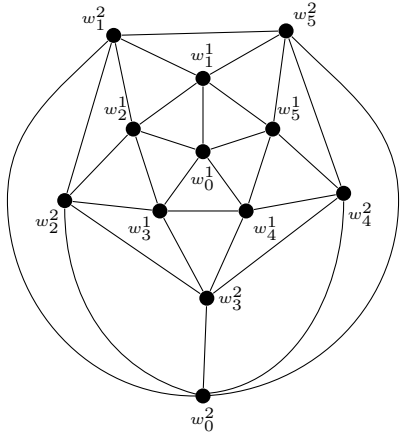
\begin{figure}[ht]
\begin{centering}
\begin{tikzpicture}[scale=1.3]
\draw (0,0) node[vx, label={[label distance=-.05cm]below:$\scriptscriptstyle{w_0^1}$}] (c) {};
\draw (0,0)+(18+72*1:.75 cm) node[vx, label={[label distance=-.1cm]18+72*1:$\scriptscriptstyle{w_{1}^1}$}] (c1) {};
\draw (0,0)+(18+72*2:.75 cm) node[vx, label={[label distance=-.2cm]18+72*2:$\scriptscriptstyle{w_{2}^1}$}] (c2) {};
\draw (0,0)+(18+72*3:.75 cm) node[vx, label={[label distance=-.2cm]18+72*3:$\scriptscriptstyle{w_{3}^1}$}] (c3) {};
\draw (0,0)+(18+72*4:.75 cm) node[vx, label={[label distance=-.2cm]18+72*4:$\scriptscriptstyle{w_{4}^1}$}] (c4) {};
\draw (0,0)+(18+72*5:.75 cm) node[vx, label={[label distance=-.2cm]18+72*5:$\scriptscriptstyle{w_{5}^1}$}] (c5) {};
\draw (0,0)+(55.5+72*1:1.5 cm) node[vx, label={[label distance=-.2cm]55.5+72*1:$\scriptscriptstyle{w_{1}^2}$}] (d1) {};
\draw (0,0)+(55.5+72*2:1.5 cm) node[vx, label={[label distance=-.2cm]55.5+72*2:$\scriptscriptstyle{w_{2}^2}$}] (d2) {};
\draw (0,0)+(55.5+72*3:1.5 cm) node[vx, label={[label distance=-.1cm]359:$\scriptscriptstyle{w_{3}^2}$}] (d3) {};
\draw (0,0)+(55.5+72*4:1.5 cm) node[vx, label={[label distance=-.2cm]55.5+72*4:$\scriptscriptstyle{w_{4}^2}$}] (d4) {};
\draw (0,0)+(55.5+72*5:1.5 cm) node[vx, label={[label distance=-.2cm]55.5+72*5:$\scriptscriptstyle{w_{5}^2}$}] (d5) {};
\foreach \x in {1,...,5}
    {
	\draw (c)--(c\x);
	\draw (c\x)--(d\x);}
\draw (c1)--(c2)--(c3)--(c4)--(c5)--(c1);
\draw (c1)--(d5);
\draw (c2)--(d1);
\draw (c3)--(d2);
\draw (c4)--(d3);
\draw (c5)--(d4);
\draw (d1)--(d2)--(d3)--(d4)--(d5)--(d1);
\draw (0,-2.5) node[vx, label={[label distance=-.1cm]below:$\scriptscriptstyle{w_0^2}$}] (d) {};
\draw (d)--(d3);
\draw (d2) to[out=-90, in=165] (d) to[out=0, in=-90] (d4);
\draw (d1) to[out=224, in=90] (-2,-.5) to[out=-90, in=180] (d) to[out=0, in=-90] (2,-.5) to[out=90,in=-45] (d5);
\end{tikzpicture}\caption{The icosahedron graph.} \label{fig:icos}
\end{centering}
\end{figure}

\begin{prop}\label{prop:C4indsat}
For $j \in \{5,6,7\}$, and $k \geq 2$, $I_j^k$ is $C_4$-induced-saturated.
\end{prop}
\begin{proof}
We first show that $I_j^k$ does not contain an induced $C_4$.  Suppose to the contrary that it does.  Since a single wheel does not contain an induced $C_4$, this $C_4$ must contain vertices from at least two different wheels.  Suppose that $w_0^p$ is in this $C_4$.  Recall that $w_0^p$ is the center of wheel $W^p$.  Then, this $C_4$ must contain $w_r^p$ and $w_s^p$ such that $|s - r| \ge 2$.  However, since $|s - r| \ge 2$, $w_r^p$ and $w_s^p$ contain no common neighbors outside of $W^p$.  Thus, all four vertices of this induced $C_4$ must be inside of $W^p$, a contradiction.  So our induced $C_4$ contains no centers of wheels.

If this $C_4$ contains exactly three vertices from a single $W^p$, then they must be consecutive along their cycle.  That is, $C_4$ contains $w_s^p, w_{s+1}^p$, and $w_{s+2}^p$.  However, as above, $w_s^p$ and $w_{s+2}^p$ have no common neighbors outside of $W^p$.  Thus, our induced $C_4$ contains at most two vertices from each $W^p$.

If this $C_4$ contains exactly two vertices from a single $W^p$, then by the same arguments used above, they must be adjacent in $W^p$, say $w_s^p$ and $w_{s+1}^p$.  No vertex of the form $w_s^q$, with $q < p$, or $w_{s+1}^r$, with $r > p$, can be in our $C_4$, as either produces a triangle with $w_s^p$ and $w_{s+1}^p$.

Now, $w_{s+1}^p$ must have another neighbor in our $C_4$.  Suppose it is in $W^t$.  If $t > p$, then it must be $w_{s+2}^t$ by the above.  However, the only common neighbors $w_{s+2}^t$ and $w_s^p$ have are of the form $w_{s+1}^q$ where $q > p$, a contradiction.  So $t < p$, and the other neighbor of $w_{s+1}^p$ is $w_{s+1}^t$.  Again though, the only common neighbors of $w_s^p$ and $w_{s+1}^t$ are either of the form $w_{s+1}^q$ where $q > p$, or $w_s^r$ where $r < p$.  In either case, we have a contradiction to the above.  Thus, our $C_4$ has exactly one vertex from each wheel.

Supose our induced $C_4$ contains the vertices $w_{t_1}^p, w_{t_2}^q, w_{t_3}^r, w_{t_4}^s$.  If $|\{t_1,t_2,t_3,t_4\}| \le 2$, then we have a triangle, a contradiction.  If $|\{t_1,t_2,t_3,t_4\}| = 4$, then some vertex is not adjacent to two of the others, a contradiction.  So $|\{t_1,t_2,t_3,t_4\}| = 3$, and two vertices have the same subscript.  We may assume that it is $w_{t_1}^p, w_{t_2}^q = w_{t_1}^q$, and that $p < q$.  Then, $w_{t_1}^p$ must have a neighbor not adjacent to $w_{t_1}^q$ in this $C_4$, say it is $w_{t_3}^r$.  However, in order for this to be possible, we must have $t_3 = t_1 + 1$ and $p < r < q$.  Thus, $w_{t_4}^s$ is adjacent to both $w_{t_1+1}^r$ and $w_{t_1}^q$.  However, since $t_4$ must be distinct from both $t_1$ and $t_1 + 1$, this cannot happen, a contradiction.  So $I_j^k$ is $C_4$-free. 

By inspection we see that $I_j^k$ has the property that every edge is the lone diagonal of a $C_4$.  Thus, removing any edge results in an induced $C_4$.  So we only need to consider adding edges.  Adding an edge within one wheel (say $W^m$) is simply adding a chord $w_i^m w_p^m$ to a 5-, 6-, or 7-cycle. If $p \neq i+2$ or $j = 5$, then this chord creates an induced 4-cycle. If $p = i+2$ and $j = 6$ or $j =7$, then if $m \neq k$, $w_i^m w_{i+1}^\ell w_{i+2}^\ell w_{i+2}^m w_i^1$ is an induced 4-cycle, where $\ell > m$, and if $m = k$, then $w_i^m w_{i+1}^m w_{i+1}^\ell w_i^\ell$ is an induced $C_4$. 

Now suppose we add an edge between wheels, say $W^m$ and $W^\ell$, where we may assume $m < \ell$. If the new edge is between the centers of these wheels, that is, $w_0^m w_0^\ell$, then $w_0^m w_0^\ell w_1^\ell w_1^m w_0^m$ is an induced $C_4$. If it is from the center of $W_m$ to a vertex on the cycle of $W^\ell$, say $w_i^\ell$, then $w_0^m w_i^\ell w_{i+1}^\ell w_{i+1}^m w_0^m$ is an induced $C_4$; a similar cycle is also created if the new edge is $w_0^\ell w_i^m$. 
Finally, if we add an edge $w_i^m w_p^\ell$, note that $w_i^m$ is not adjacent to at least one of $w_p^m$ and $w_{p-1}^m$; label this vertex $u$. Since $u$ is adjacent to $w_{p}^\ell$, the vertices $w_0^m, w_i^m, w_p^\ell$, and $u$ induce a $C_4$. 
\end{proof}

Proposition \ref{prop:C4indsat} implies that for many values of $n$, $\indsat{n}{C_4}=0$. In fact, this is the case for $n \geq 12$. To show this, we use the following proposition regarding $kK_2$.  While we only employ the proposition in the case $k = 2$, the more general statement which we present is not difficult.

\begin{prop}\label{twins matching}
Let $s:=(s_1,\dots, s_n)$ be a sequence of positive integers.  Let $G$ be a graph with vertex set $\{v_1,\ldots v_n\}$, and let $G_s$ be the graph obtained from $G$ by replacing each vertex $v_i$ with an independent set of order $s_i$ and each edge with a complete bipartite graph between the corresponding independent sets. 
For $k \ge 2$, $G$ is $kK_2$-induced-saturated if and only if $G_s$ is  $kK_2$-induced-saturated. 
\end{prop}

\begin{proof}
For each vertex $v_i \in V(G)$, let $V_i$ be the independent set in $G_s$ that corresponds to it. 
We will call this collection of vertices in $G_s$ that replaces a single vertex in $G$ a \emph{part}. 

Note that 
no induced matching in $G_s$ uses two vertices from the same part, and the same holds if we add or remove a single edge from $G_s$.
We claim that if $w_i$ and $w_j$ are vertices from different parts $V_i$ and $V_j$, respectively, of $G_s$, then $G_s$ (or  $G_s+w_iw_j$ or $G_s-w_iw_j$) contains an induced matching 
if and only if $G$ (resp. $G+v_iv_j$, or $G-v_iv_j$) contains an induced matching $M$. 
Suppose $M_s$ is such an induced matching in $G_s$ (or $G_s+w_iw_j$ or $G_s-w_iw_j$). Then each vertex in $M_s$ comes from a different part of $G_s$ (resp. $G_s+w_iw_j$ or $G_s-w_iw_j$), and thus they correspond to distinct vertices in $V(G)$. 
This is an induced matching in $G$.

If $G$ (or $G+v_iv_j$ or $G-v_iv_j$) has an induced matching $M$, then when the graph is expanded, no new adjacencies have been added between the parts corresponding to the endpoints of vertices in $M$ (except for $w_iw_j$ in the case of $G+v_iv_j$).  Thus, we can find an induced matching in $G_s$ (resp. $G_s+w_iw_j$ or $G_s-w_iw_j$).  
This shows that if $G_s$ is $kK_2$-induced-saturated, then so is $G$.

To show that if $G$ is $kK_2$-induced-saturated, then so is $G_s$, it remains to consider adding edges between vertices in one part of $G_s$.  First we note that $G$ has no dominating vertex. Indeed, if $u$ is a dominating vertex, then deleting an edge incident to $u$, say $uw$, does not create an induced $2K_2$, let alone an induced $kK_2$, as $u$ dominates $N_G(w)$.  

Now, suppose we add $w_iw_i'$ to $G_s$, in the part $V_i$ corresponding to $v_i$. 
Since $v_i$ is not dominating, there exists $w$ not adjacent to $v_i$.  Since $G$ is $kK_2$-induced-saturated, $G+v_iw$ contains an induced matching $M=\{v_iw,x_2y_2,\ldots,x_ky_k\}$. Then $M_s=\{w_iw_i',X_2Y_2,\ldots,X_kY_k\}$ is an induced matching in $G_s+w_iw_i'$, where $X_j$ and $Y_j$ are vertices in the parts corresponding to $x_j$ and $y_j$, respectively.
\end{proof}

\begin{cor}
For $n \ge 12$, $\indsat{n}{C_4} = 0$.
\end{cor}

\begin{proof}
Applying Observation \ref{thm:complement} to case $k = 2$ in Proposition \ref{twins matching}, allows us to begin with a graph that is $C_4$-induced-saturated, replace a single vertex with a clique of any order, replace the affected edges with complete bipartite graphs, and produce another graph that is $C_4$-induced-saturated.  Thus, beginning with $I_5^2$, applying these operations obtains $C_4$-induced-saturated graphs for all values of $n \ge 12$.
\end{proof}

For $4 \le n \le 10$, a computer search showed $\indsat{n}{C_4} > 0$.  At this time, whether $\indsat{11}{C_4}$ is zero or not, is yet unknown.  We now turn our attention to $\sis{n}{C_4}$.

\begin{thm}\label{prop:C4}
For sufficiently large $n$, $(5/2)n \leq \sis{n}{C_4}\leq (7/64)n^2+o(n)$.
\end{thm}
\begin{proof}
To prove the lower bound we show that $\delta(G) \ge 5$.
Suppose $G$ is a $C_4$-induced-saturated graph. 
Let $x \in V(G)$, and let $H:=G[N(x)]$.  Since deleting any edge produces an induced $C_4$, every edge is the diagonal of a $C_4$ and $\deg(x) \ge 3$.  In particular, there exists $v_1,v_2,v_3 \in V(H)$ such that $v_1v_3$ is not an edge, but $v_1v_2$ and $v_2v_3$ are edges.  Now, $G - xv_1$ contains an induced $C_4$ that contains both $x$ and $v_1$, but not $v_3$.  If $v_2$ is not in this $C_4$, then there exists two other vertices distinct from $v_1,v_2,v_3$ in $H$.  Thus, $\deg(x) \ge 5$.  If $v_2$ is in this $C_4$, then there exists $v_4 \in V(H)$  distinct from $v_1,v_2,v_3$ such that $v_1v_4$ is an edge, but $v_2v_4$ is not.  By a similar argument, considering $G - xv_3$ gives at least one additional vertex in $H$ distinct from $v_1,v_2,v_3,v_4$.  So in any case, $\deg(x) \ge 5$, and as $x$ was arbitrary, $\delta(G) \ge 5$.  Thus, provided $n \ge 12$, $\sis{n}{C_4} \ge (5/2)n$.


To prove the upper bound, we choose $n \geq 56$ and create a graph $G$ of order $n$. Let $r \equiv n \mod 8$, where $0 \le r \le 7$.  Set $k=\lfloor n/8 \rfloor$ so that $k \ge r$ and $n(I_7^k)=8k$.  If $r=0$, choose $G=I_7^k$. If $r>0$, we create $G$ by adding $r$ vertices to $I_7^k$. Recall, as discussed after Proposition~\ref{twins matching},
by replacing the vertices of $I_7^k$ with cliques, and its edges with complete bipartite graphs,
we preserve the property of being $C_4$-induced-saturated.
Accordingly, using the notation of Construction~\ref{gen gen icos}, we replace $w_0^1,\ldots, w_0^r$ with copies of $K_2$ and make each new vertex adjacent to the neighborhood of the vertex it replaces.

Now we determine $e(G)$. The first $r$ wheels have $22$ edges, and the rest have $14$. Between any two wheels there are $14$ edges. So $e(G)=14\left[{k \choose 2}+k\right]+8r$. Since
$r \in [0,7]$ and $k=\lfloor n/8 \rfloor$, $e(G) \leq \frac{7}{64}n^2+\frac{7}{8}n+56$.

\end{proof}

\subsection{Matchings}
Another graph that is $C_4$-induced-saturated is the join $I_5^2 \vee K_{n-12}$. 
Observation \ref{thm:complement} implies that the complement of this graph is $2K_2$-induced-saturated. We can further generalize this to get a $kK_2$-induced-saturated graph for any $k\geq 2$. 

\begin{prop}\label{clm:match}
Let $\overline{I_5^2}$ be the complement of the icosahedron. For fixed $k$ and $n \geq 12(k-1)$, the graph $(k-1)\overline{I_5^2} + (n-12(k-1))K_1$ is $kK_2$-induced-saturated. Thus, for $n \geq 12(k-1)$, $\indsat{n}{kK_2}=0$.
\end{prop}
\begin{proof}
By Proposition~\ref{prop:C4indsat} and Observation~\ref{thm:complement}, the complement of an icosahedron is $2K_2$-induced-saturated. Let $G$ denote $(k-1)\bar{I_5^2} + (n - 12(k - 1))K_1$.  Clearly, $G$ contains $(k-1)K_2$ as an induced subgraph, but no induced $kK_2$.  If we add or delete any edge inside a component, or add an edge among the isolates, we create an induced $kK_2$. Note that every vertex $v$ in $\bar{I_5^2}$ is in an induced copy of $K_2+K_1$ where $v$ is the isolate.  Thus, adding any edge with an endpoint in a copy of $\bar{I_5^2}$ creates an induced $kK_2$.
\end{proof}


\begin{cor} For $n\geq 12(k-1)$, $\sis{n}{kK_2} \leq 36(k-1)$.
\end{cor} 
In particular, for fixed $k$, $\sis{n}{kK_2}$ is constant.

%
%

\section{Other Cycles and Generalizations of Cycles}

In this section we provide a construction proving that odd cycles also have induced saturation number zero for $n$ sufficiently large.  As it is already known that  $\indsat{n}{C_3} = \sat{n}{C_3}$ \cite{MS}, we only consider odd cycles of length at least five.  Additionally,  this construction is also $H$-induced-saturated when $H$ is a modification of an even cycle as described below.

Let $C_{2k}'$ denote a cycle of length $2k$ with a pendant vertex, and $\hat C_{2k}$ denote an even cycle with a chord between two vertices at distance 2 from each other (sometimes called a triangle chord or hop).

For a given $k$ and $n \geq (k+1)^2+2$, we can write $n$ as $(k+1)t-s$ where $t$ and $s$ are integers with $t \geq k+2$ and $0 \leq s \leq t-3$. 
In particular, we choose $t=\cl{\frac{n}{k+1}}$. 
Using this expression for $n$, we give the following construction.

\begin{cons}\label{cons:cycles} 
For $k \geq 3$ and $n \geq (k+1)^2+2$, let $n=(k+1)t-s$, where $t=\cl{\frac{n}{k+1}} \ge k+2$ and $0 \leq s \le t-3$. 
Let $G_{n,k}$ be formed from the Cartesian product $K_{k+1} \cartprod K_t$ by removing $s$ vertices from one copy of $K_t$. 
\end{cons}

\begin{prop}\label{indsat_cycles}
If $H \in \{C_{2k-1}, C_{2k}', \hat C_{2k}\}$, then the graph $G_{n,k}$ in Construction \ref{cons:cycles} is $H$-induced-saturated.
\end{prop}
\begin{proof}
Let $G_{n,k}$ be as described in Construction \ref{cons:cycles}.  We first show that $G_{n,k}$ is $H$-free for $H \in \{C_{2k-1}, C_{2k}', \hat C_{2k}\}$.  Any induced subgraph of $G_{n,k}$ that is triangle-free has at most two vertices from any copy of $K_{k+1}$ or $K_t$. 
Since $2k-1$ is odd, an induced $C_{2k-1}$ would contain precisely one vertex $v$ from some copy of $K_{k+1}$. Then the neighbors of $v$ must be in the same copy of $K_t$, which means they form a triangle. Thus, $G_{n,k}$ has no induced odd cycle larger than a triangle.  Since $\hat C_{2k}$ contains $C_{2k-1}$ as an induced subgraph, neither $C_{2k-1}$ nor $\hat C_{2k}$ are induced subgraphs of $G_{n,k}$.  Similarly, if $G_{n,k}$ contained an induced $C'_{2k}$, then because $C'_{2k}$ is triangle-free with an odd number of vertices, then there would be one copy of $K_t$ that contains precisely one vertex $v$ of the subgraph. If $v$ is on the cycle, it has at least two neighbors, but these can only be other copies of $v$, forming a triangle in some copy of $K_{k+1}$. If $v$ is the pendant vertex, suppose it has neighbor $u$ on the cycle. Then $u$ has some neighbor $u'$ in a different copy of $K_t$ from itself, and $u,u'$, and $v$ are all in one copy of $K_{k+1}$, forming a triangle. Thus, $G_{n,k}$ has no induced $C'_{2k}$.

In the remainder of this proof we view $K_{k+1} \cartprod K_t$ as a $k+1$ by $t$ grid with vertices $v_{i,j}$ for $1 \le i \le k+1$ and $1 \le j \le t$, and the graph induced by the vertices a row (or column) is a clique.  In particular, we have $k+1$ columns and $t$ rows.  We then form $G_{n,k}$ by removing $s$ vertices from a copy of $K_t$, and let $K_t^*$ denote the clique induced by the remaining $t - s$ vertices.   
Since $s \le t-3$ and $k \ge 3$, $G_{n,k}$ has at least three vertices in each row and column, and at least three rows with at least 4 vertices.  

Let $G_{n,k}'$ be obtained from $G_{n,k}$ by adding an edge.  Up to relabeling, we may assume that $v_{1,1}v_{k+1,k+1}$ is the added edge, and $v_{1,1}$ is not in $K_t^*$.  Suppose first that $v_{k+1,k+1}$ is in $K_t^*$ (i.e., the $(k+1)$st column induces $K_t^*$).  Let $T := \{v_{i,i+1}, v_{i+1,i+1}: 1 \le i \le k-2\} \cup \{v_{1,1}, v_{k+1,k+1}\}$.  Then $T \cup \{v_{k-1,k+1}\}$ induces $C_{2k-1}$, $T \cup  \{v_{k-1,k},v_{k-1,k+1}\}$ induces $\hat C_{2k}$, and $T \cup \{v_{k-1,k}, v_{k,k}, v_{k,1}\}$ induces $C_{2k}'$.

If $v_{k+1,k+1}$ is not in $K_t^*$, then we can switch the column containing $K_t^*$ with the $k$th column and relabel the vertices appropriately.  Note that $v_{1,1}v_{k+1,k+1}$ is still the added edge as $K_t^*$ is not in the first or $(k+1)$st column.
  In this case, we find an induced $C_{2k-1}$ and $\hat C_{2k}$ from the same sets used above.  In order to find an induced $C_{2k}'$, we consider the vertices in $K_t^*$.  Since $t - s \ge 3$, $K_t^*$ contains a vertex not in the first or $(k+1)$st column.  Up to permuting rows 2 through $k$ and relabeling vertices, we may assume that $v_{k,k}$ is such a vertex.  If $v_{k,k+1}$ is also a vertex in $K_t^*$, then $(T\setminus\{v_{1,2}\}) \cup \{v_{k+1,2}, v_{k,k+1}, v_{k,k}, v_{k-1,k}\}$ induces $C_{2k}'$.  If not, then $K_t^*$ contains two vertices not in the first or $(k+1)$st row, and up to permuting rows 2 through $k-1$ and relabeling vertices, we may assume that $v_{k,k-1}$ is such a vertex.  Then $\{v_{i,i}, v_{i+1,i}: 2\le i \le k\} \cup \{v_{2,k+1}, v_{k+1,k+1}, v_{1,1}\}$ induces $C_{2k}'$.

Now let $G_{n,k}'$ be obtained from $G_{n,k}$ by deleting an edge.  Again, up to relabeling, we may assume that either $v_{1,1}v_{1,2}$  or $v_{1,1}v_{2,1}$ is the deleted edge.  

Suppose first that $v_{1,1}v_{1,2}$ is the deleted edge.  If this edge was in $K_t^*$, we can switch rows and relabel the vertices so that $v_{1,1}, v_{1,2}$, and $v_{1,k+1}$ are in $K_t^*$.  If not $v_{1,1}v_{1,1}$ is not in $K_t^*$, then switch the column containing $K_t^*$ with the $(k+1)$st column and relabel the vertices.  Note that $v_{1,1}v_{1,2}$ is still the deleted edge.  For $T$ as defined above, $T \cup \{v_{k-1,1}\}$ induces $C_{2k-1}$, $T \cup \{v_{k-1,1}, v_{k-1,k}\}$ induces $\hat C{2k}$, and $T \cup \{v_{k-1,k}, v_{k,k}, v_{k,k+1}\}$  induces $C_{2k}'$.  

Now suppose that $v_{1,1}v_{2,1}$ is the deleted edge, and without loss of generality,  $v_{1,1}$ is not in $K_t^*$.  If $v_{2,1}$ is in $K_t^*$, then we can switch rows and relabel vertices so that $v_{2,2}$ exists.  If $v_{2,1}$ is not in $K_t^*$, then we switch the column containing $K_t^*$ with the $(k+1)$st column.  Note that $K_t^*$ contains at least two vertices not in the first row so that up to permuting rows 2 through $k+1$ and relabeling vertices, we may assume $v_{k+1,k-1}$ and $v_{k+1,k}$ exist.  In either case, the deleted edge is still $v_{1,1}v_{2,1}$.  Let $T' := \{v_{i+1,i}, v_{i+1,i+1}: 1 \le i \le k - 2\} \cup \{v_{1,1}\}$.  Then $T \cup \{v_{1,k-1}, v_{k,1}\}$ induces $C_{2k-1}$, $T \cup \{v_{1,k-1}, v_{1,k}, v_{k,1}\}$ induces $\hat C_{2k}$, and $T \cup \{v_{k+1,k-1}, v_{k+1,k}, v_{k,k}, v_{k,1}\}$ induces $C_{2k}'$.
%
%
%
%
\end{proof}

\begin{cor}\label{cor:cycles}
For all $k \geq 3$, if $n \geq (k+1)^2+2$ and $H\in\{C_{2k-1},C_{2k}',\hat C_{2k}\}$, then $\indsat nH =0$.
\end{cor}

In the following discussion assume $H \in \{C_{2k-1}, C_{2k}', \hat C_{2k}\}$.  Using Construction \ref{cons:cycles} we obtain an upper bound on $\sis{n}{H}$ with order of magnitude $n^2$, which is trivial.  We can improve this order of magnitude slightly in the case when $\lceil \sqrt{n}~ \rceil$ is not prime.  To do so we note that if $n$ can be written as a product of two integers $s$ and $t$ that are both at least $k$, then the graph $K_s \cartprod K_t$ is $H$-induced-saturated.

\begin{prop}
Fix $k \ge 3$ and choose $n$ such that $n^{1/4} \ge k + 1$.  For $H \in \{C_{2k-1}, C_{2k}', \hat C_{2k}\}$, if $\lceil \sqrt{n}~ \rceil$ is divisible by some $t \ge 3$, $\sis{n}{H} \le cn^{7/4} + O(n^{3/2})$ for some constant $c$.
\end{prop}

\begin{proof}
As noted above, the Cartesian product of two sufficiently large cliques is $H$-induced-saturated.  So, consider $G := K_{\lceil \sqrt{n}~\rceil/t} \cartprod K_{t\lceil \sqrt{n}~ \rceil}$.  Simple computation shows $n \le v(G) \le n + 2\sqrt{n} + 1$.  So, $n(G)$ can be written as $n + s$, where $0 \le s \le 2\sqrt{n} + 1 \le t\sqrt{n} - 3$, as $t \ge 3$.  Let $G'$ be obtained from $G$ by removing $s$ vertices from a single copy of $K_{3\lceil \sqrt{n}~\rceil}$ as in Construction \ref{cons:cycles}.  An argument similar to that in Proposition \ref{indsat_cycles} shows that $G'$ is $H$-induced-saturated.  Observe: $$e(G') \le t\lceil\sqrt{n}~\rceil \binom{(1/t)\lceil \sqrt{n}~\rceil}{2} + (1/t)\lceil\sqrt{n}~\rceil \binom{t\lceil \sqrt{n}~\rceil}{2} = \frac{\lceil \sqrt{n}~\rceil^2}{2}\left(\left(t+\frac{1}{t}\right)\lceil\sqrt{n}~\rceil) - 2\right).$$
Since $t$ divides $\lceil\sqrt{n}~\rceil$, $t \le \sqrt{\lceil \sqrt{n}~\rceil}\le c'n^{1/4}$ for some $c' > 1$.  Using this and $\lceil\sqrt{n}~\rceil \le \sqrt{n} + 1$ gives $e(G') \le \frac{c'}{2}n^{7/4} + O(n^{3/2})$.
\end{proof}

Considering odd cycles  points out another property of the induced saturation number. 
That is, if $\indsat{n}{H}=0$ for a particular $n$, it is not necessarily the case that $\indsat{k}{H}=0$ for all $k > n$. For example,  Construction~\ref{cons:cycles} shows $\indsat{n}{C_5}=0$ for $n = 9$ and $n \geq 12$. 
However, a computer search showed that for $n=10$ and $n=11$, we have $\indsat{n}{C_5} > 0$. 
(A $C_5$-induced-saturated trigraph on 10 vertices with one gray edge is shown in Figure \ref{fig:C5_10}, so that $\indsat{10}{C_5}=1$.) 



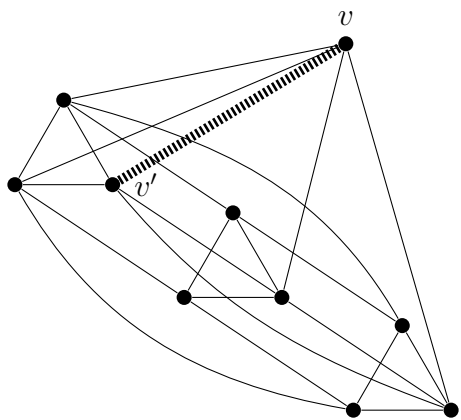
\begin{figure}[h]
\centering
\begin{tikzpicture}[scale=0.75]
\foreach \y in {1,2,3}
	{\foreach \x in {1,2,3}
	{\draw (-3*\y,2*\y)+(90+120*\x:1cm) 
	node[vx] (t\y\x){};}
	\draw (t\y1)--(t\y2)--(t\y3)--(t\y1);}
\draw (t11)--(t21)--(t31);
\draw (t31) to[out=-60, in=170] (t11);
\draw (t13)--(t23)--(t33);
\draw (t33) to[out=-15, in=120] (t13);
\draw (t32)--(t22)--(t12);
\draw (t32) to[out=-50, in=160] (t12);
\draw (-4,8) node[vx, label=above:$v$] (v){};
\draw (t31)--(v)--(t33);
\draw (t22)--(v)--(t12);
\draw (t32) node[label= right:$v'$]{};
\draw [line width=3.6pt,dash pattern=on 1pt off 1pt]  (t32) -- (v);
\end{tikzpicture}
\caption{This trigraph, with the gray edge $vv'$, is a $C_5$-induced-saturated trigraph.}
\label{fig:C5_10}
\end{figure}

%
%

\section{Families of Graphs}\label{sec:families}

In this section we extend the definition of induced saturation in \cite{MS} to families of graphs in the natural way. 

\begin{defn}
For a family $\mathcal F$ of graphs, a trigraph $T$ is \textbf{$\mathcal F$-induced-saturated} if no realization of $T$ contains any member of $\mathcal F$ as an induced subgraph, but whenever any black or white edge of $T$ is turned to gray, some member of $\mathcal F$ occurs as an induced subgraph of some realization. 

The \textbf{induced saturation number} of $\mathcal F$ with respect to $n$, written $\indsat{n}{\mathcal F}$, is the minimum number of gray edges in an $\mathcal{F}$-induced-saturated trigraph with $n$ vertices.
\end{defn}

For any family $\mathcal F$ containing all graphs on $k$ vertices, 
 $\indsat{n}{\mathcal F}={n \choose 2}$.

Construction~\ref{cons:cycles} and Proposition~\ref{indsat_cycles} demonstrate that for any family $\mathcal F$, all of whose elements are odd cycles, even cycles with a pendant, or even cycles with a triangle chord, $\indsat{n}{\mathcal F}=0$ for $n$ sufficiently large. 
However, we could have $\indsat{n}{\mathcal{F}}\neq 0$ even if there is some $G\in \mathcal{F}$ such that $\indsat{n}{G}=0$ 
as demonstrated in Proposition~\ref{threshold} below. 
One may suspect this is because of the presence of $P_4$, which has nonzero induced-saturation number, yet
it is also possible for a family $\mathcal F$ to consist of graphs that each individually have induced saturation number zero, while the induced saturation number of $\mathcal F$ is nonzero. We provide an example of this in Proposition~\ref{split}.

\begin{prop}\label{threshold}
For all $n$, $\indsat{n}{\{2K_2,P_4,C_4\}} \neq 0$.
\end{prop}
\begin{proof}
The graphs that contain no induced $2K_2,P_4,$ or $C_4$ are precisely the \emph{threshold graphs} \cite{CH}. These graphs are characterized in a second way: they are constructed by iteratively adding a vertex to a graph either as an isolate or a dominating vertex. Thus, an $n$-vertex threshold graph can be represented as a string of $n$ symbols from $\{-,+\}$ as follows:
on the vertex set $V=\{v_1,\ldots,v_n\}$, for every $i>j$, $v_iv_j$ is an edge if and only if the $i$th symbol in the string is $+$.

We claim that for any threshold graph $G$ with at least one edge, there exists $e \in E(G)$ such that $G-e$ is also threshold. Let $\pi=s_1, \ldots, s_n$ be a string of symbols from $\{-,+\}$ representing $G$. Suppose there exists $i \in [n-1]$ such that $s_i=-$ and $s_{i+1}=+$, and let $i$ be minimal with this property. Then the graph $G'=G-v_iv_{i+1}$ is represented by the symbol list 
$\pi'=s_1 \ldots s_{i-1}s_{i+1}s_is_{i+1}\ldots s_n$, so $G'$ is threshold.
If no such index $i$ exists, then $\pi$ is a list consisting only of $+$, so $G$ is the complete graph $K_n$; however, $K_n-e$ is also threshold.

Thus, for any graph $G$ with no induced $2K_2$, $P_4$, or $C_4$, there exists an edge $e \in G$ such that $G-e$ also has no induced $2K_2$, $P_4$, or $C_4$.
It follows that $\indsat{n}{\{2K_2,P_4,C_4\}} \neq 0$. 
\end{proof}

The family of {\it split graphs} is another family of graphs that can be characterized by a set of forbidden induced subgraphs. A split graph is a graph whose vertex set can be partitioned into a clique and an independent set. F\"oldes and Hammer~\cite{FH} showed that a graph is a split graph if and only if it contains no induced $2K_2$, $C_4$, or $C_5$.

\begin{prop}\label{split}
For all $n$, $\indsat{n}{\{2K_2,C_4,C_5\}} \neq 0$.
\end{prop}
\begin{proof}
Since adding or deleting an edge between the clique part and the independent set of a split graph still results in a split graph, it follows that $\indsat{n}{\{2K_2, C_4, C_5\}} \neq 0$.
\end{proof}

We have shown that $\indsat{n}{2K_2}$, $\indsat{n}{C_4}$, and $\indsat{n}{C_5}$ are all equal to zero for sufficiently large $n$. Thus, this example shows that even though every graph in a family has induced-saturation number zero, the family itself may not have induced-saturation number zero. 

Other families characterized by a (not necessarily finite) family of forbidden induced subgraphs include 
perfect graphs \cite{CRST}, 
trivially perfect graphs \cite{W}, \cite{G}, 
interval graphs \cite{LB}, 
and line graphs \cite{B}.
It would be interesting to determine $\indsat{n}{\mathcal F}$ and $\sis{n}{\mathcal F}$ for these families. We suspect that doing so will be much more difficult than for threshold and split graphs, as the families of forbidden graphs are significantly more complicated.

\section{Acknowledgements}
The authors would like to thank Michael Ferrara and Douglas West for their kind support and advice. 

\end{document}